\documentclass[12pt,reqno]{amsart}

\usepackage[dvips]{graphicx}
\usepackage{amssymb, latexsym, amsmath, amscd, array, amsthm, epsfig, pstricks}
\usepackage{yhmath}
\usepackage[all]{xy}
\usepackage{tikz-cd}

\usepackage{hyperref}

\usepackage{breakurl}   

\newcommand{\G}{{\mathcal G}}

\newcommand{\R}{{\mathbb R}}

\newcommand{\C}{{\mathbb C}}

\newcommand{\area}{{\rm area}}

\newcommand{\sys}{{\rm sys}}

\newcommand{\ie}{{\it i.e.}}
\newcommand{\eg}{{\it e.g.}}
\newcommand{\Capacity}{{\rm Cap}}

\DeclareMathOperator\arsinh{arcsinh}

\numberwithin{equation}{section}

\newtheorem{theorem}{Theorem}[section]
\newtheorem{proposition}[theorem]{Proposition}
\newtheorem{corollary}[theorem]{Corollary}
\newtheorem{lemma}[theorem]{Lemma}

\theoremstyle{definition}
\newtheorem{definition}[theorem]{Definition}
\newtheorem{example}[theorem]{Example}
\newtheorem{remark}[theorem]{Remark}

\long\def\forget#1\forgotten{}

\begin{document}

\title[Systolically extremal surfaces] {Systolically extremal nonpositively curved surfaces are flat with finitely many singularities}

\subjclass[2010]{Primary 53C23; Secondary 53C20.}

\keywords{systole, systolic inequalities, extremal metrics, nonpositively curved metrics, piecewise flat surfaces with conical singularities}

\author[M.~Katz]{Mikhail G. Katz} \address{M.~Katz, Department of
Mathematics, Bar Ilan University, Ramat Gan 52900 Israel}
\email{katzmik@macs.biu.ac.il}

\author[S.~Sabourau]{St\'ephane Sabourau} \address{S.~Sabourau,
Universit\'e Paris-Est, Laboratoire d'Analyse et Math\'ema\-tiques
Appliqu\'ees (UMR 8050), UPEC, UPEMLV, CNRS, F-94010, Cr\'eteil,
France}\email{stephane.sabourau@u-pec.fr}

\begin{abstract}
The regularity of systolically extremal surfaces is a notoriously
difficult problem already discussed by M. Gromov in 1983, who proposed
an argument toward the existence of~$L^2$-extremizers exploiting the
theory of~$r$-regularity developed by P. A. White and others by the
1950s.  We propose to study the problem of systolically extremal
metrics in the context of generalized metrics of nonpositive
curvature.  A natural approach would be to work in the class of
Alexandrov surfaces of finite total curvature, where one can exploit
the tools of the completion provided in the context of Radon measures
as studied by Reshetnyak and others.  However the generalized metrics
in this sense still don't have enough regularity.  Instead, we develop
a more hands-on approach and show that, for each genus, every
systolically extremal nonpositively curved surface is piecewise flat
with finitely many conical singularities.  This result exploits a
decomposition of the surface into flat systolic bands and nonsystolic
polygonal regions, as well as the combinatorial/topological estimates
of Malestein--Rivin--Theran, Przytycki, Aougab--Biringer--Gaster and
Greene on the number of curves meeting at most once, combined with a
kite excision move.  The move merges pairs of conical singularities on
a surface of genus~$g$ and leads to an asymptotic upper bound
$g^{4+\epsilon}$ on the number of singularities.
\end{abstract}

\maketitle

\tableofcontents

\section{Introduction}
\label{one}

The \emph{systole} of a Riemannian manifold~$M$, denoted~$\sys(M)$, is
the least length of a noncontractible loop in~$M$.  A seminal text in
this area is Gromov's paper \emph{Filling Riemannian
manifolds}~\cite{Gr83}.  It deals in particular with the problem of
the existence of systolically extremal surfaces, \ie, surfaces with
maximal systole for a fixed area, or equivalently minimal area for a
fixed systole.  There is a discussion of systolically extremal
surfaces without curvature assumptions in \cite[pp.\;64--65]{Gr83}.
The proposed existence of the surfaces in question is only in a weak
sense as it relies on the theory of~$r$-regular convergence of
P. A. White and others, introduced in the thirties; see~\cite{Wh54}.
More precisely, systolically extremal surfaces are endowed with a
length metric structure along with a (possibly vanishing)~$L^2$-limit
of the conformal factors of some approximating Riemannian metrics.
Despite this preliminary result, the existence of more regular
systolically extremal surfaces without curvature assumptions remains
an open problem, except for the torus~\cite{katz}, the projective
plane~\cite{pu} and the Klein bottle~\cite{bav}, where systolically
extremal metrics have been determined (for other optimal Loewner-type
inequalities see \cite{Ba07}, \cite{IK}, \cite{KL}).  No conjecture is
available for other surfaces, except in genus 3 where Calabi
constructed nonpositively curved piecewise flat metrics with
systolically extremal-like properties; see~\cite{cal}, \cite{sab11}
(and \cite{SY} for related systolic-like properties in genus~$2$).
Partial results concerning systolically optimal metrics were obtained
by Bryant \cite{Br96} using PDE techniques, assuming regularity.

\subsection{Statement of the problem}

We will study the extremality problem in the context of surfaces
endowed with a Riemannian metric of nonpositive curvature.  The
\emph{systolic area}~$\sigma$ of a surface~$M$ with a fixed metric is
defined as
\[
\sigma(M) = \frac{\area(M)}{\sys(M)^2}.
\]
The optimal systolic area in genus~$g$ for nonpositive curvature is
defined as
\begin{equation}
\label{eq:SH}
\sigma_{\mathcal{H}}^{\phantom{I}}(g) = \inf_M \sigma(M)
\end{equation}
where the infimum is taken over all nonpositively curved genus~$g$
surfaces~$M$.  Here, the subscript~$\mathcal{H}$ alludes to Hadamard
as the surfaces considered are locally CAT$(0)$.  For a recent study
of Hadamard spaces see Ba\v c\'ak \cite{Ba18}.

For surfaces of nonpositive curvature of genus~$g=2$, we showed
in~\cite{KS} that the metric realizing the
infimum~$\sigma_{\mathcal{H}}^{\phantom{I}}(2)$ of the systolic area
is flat with finitely many conical singularities, in the conformal
class of the smooth completion of the affine complex algebraic
curve~$w^2=z^5-z$, and one
has~$\sigma_{\mathcal{H}}^{\phantom{I}}(2)=3\tan(\frac{\pi}{8})$.%
\footnote{The same conformal class contains an optimal metric for a
related first eigenvalue problem; see \cite{Na17}.  This optimal
metric similarly has finitely many conical singularities.}
A similar result holds for the metric realizing the infimum of the
systolic area among all nonpositively curved metrics on the surface
homeomorphic to the connected sum of three projective planes, also
known as Dyck's surface; see~\cite{KS15}.

The purpose of the present text is to extend this result to 
surfaces of arbitrary genus.  We will need a few more definitions to
cover the case of local infima, and not just global infima.

\begin{definition} 
\label{d11}
A closed surface~$M$ with a Riemannian metric with conical
singularities is locally isometric to the complex plane endowed with
the metric
\[
ds^2 = e^{2u(z)} \, |z|^{2 \beta} \, |dz|^2
\]
where~$\beta > -1$ and~$u\colon \C \to \R$ is a continuous function,
smooth everywhere except possibly at the origin.
\end{definition}

See Troyanov~\cite{tro91} for a detailed description.  Here, the point
of~$M$ corresponding to the origin in~$\C$ is a conical singularity of
total angle \mbox{$\theta = 2\pi(\beta+1)$}.

\begin{example}
Gluing together~$n$ Euclidean angular sectors of angle~$\theta_i,
i=1,\ldots,n$ side by side in circular order gives rise to a conical
singularity of total angle~$\theta_1+\cdots+\theta_n$.
\end{example}

Such a surface~$M$ is nonpositively curved (in Alexandrov's sense) if
and only if the Gaussian curvature of~$M$ is nonpositive away from the
conical singularities and the total angle at each conical singularity
is greater than~$2 \pi$.

\begin{definition} 
The space~$\mathcal{H}_g$
consists of nonpositively curved Riemannian metrics (possibly with
conical singularities) on a genus~$g$ surface.  This space will be
endowed with either of the following nonequivalent distances, namely,
the Gromov--Hausdorff distance or the Lipschitz distance;
see~\cite{Gr99}.  It fibers over the conformal moduli
space~$\mathcal{M}_g$ (see~\cite{tro91}):
\[
\begin{tikzcd}
\mathcal{H}_g \arrow{d} \\
\mathcal{M}_g
\end{tikzcd}
\]
where the base is~$(6g-6)$-dimensional and the fiber
infinite-dimensional.
\end{definition}

\begin{definition} \label{def:local}
A \emph{local infimum} of the systolic area
on~$\mathcal{H}_g$ is a real number~$\mu > 0$ such that
there exists an open
set~$\mathcal{U}\subseteq \mathcal{H}_g$ satisfying a strict
inequality
\begin{equation}
\label{eq:localinf}
\mu = \inf_{M \in \mathcal{U}} \sigma(M) < \inf_{M \in \partial
\mathcal{U}} \sigma(M).
\end{equation}
\end{definition}

Note that, though we use the term \emph{local}, this definition is not
entirely local as the strict inequality~\eqref{eq:localinf} may hold
for some open set~$\mathcal{U}$, but fail for arbitrarily small ones.

\begin{definition}
A nonpositively curved surface~$M \in \mathcal{H}_g$
(possibly with conical singularities) is \emph{locally extremal} for the
systolic area if there exists an open
set~$\mathcal{U}\subseteq \mathcal{H}_g$ containing~$M$ such
that
\[
\sigma(M) = \inf_{M \in \mathcal{U}} \sigma(M) < \inf_{M \in \partial
\mathcal{U}} \sigma(M).
\]
In such case we say that the local infimum~$\displaystyle \mu =
\inf_{M \in \mathcal{U}} \sigma(M)$ is \emph{attained} by~$M$.
\end{definition}

\subsection{Main results}

We can now state our main result concerning the existence of
systolically extremal metrics.

\begin{theorem} 
\label{theo:main}
Every local infimum of the systolic area on the
space~$\mathcal{H}_g$ of nonpositively curved genus~$g$
surfaces (possibly with conical singularities) is attained by a
nonpositively curved piecewise flat metric with at most~$\mathcal{N}_0 \leq 2^{25} \,
g^4 \log^2(g)$ conical singularities.
\end{theorem}

\begin{remark}
The upper bound~$\mathcal{N}_0$ can be expressed in terms of the
maximal number~$\bar{Q}(g)$ of systolic homotopy classes on a closed
nonpositively curved surface of genus~$g$ with finitely many conical
singularities; see Theorem~\ref{theo:bound}.  In turn, the
quantity~$\bar{Q}(g)$ can be bounded in terms of the number~$Q(g)$ of
pairwise nonhomotopic simple closed curves on a genus~$g$ surface; see
Theorem~\ref{theo:ABG} and Proposition~\ref{item:kissing3}.
\end{remark}

Thus the number of conical singularities is uniformly bounded for all
locally extremal metrics of nonpositive curvature on a surface of
fixed genus.  By Theorem~\ref{theo:main} and the conformal
representation of piecewise flat surfaces, see~\cite[\S 5]{tro86},
every locally extremal metric has a well-defined conformal class and a
well-defined continuous conformal factor (with finitely many zeros).
Moreover, the space of locally extremal nonpositively curved metrics
is finite-dimensional of dimension at most~$2^{25} \, (6g-6) \, g^4
\log^2(g)$.

\begin{corollary} 
For every genus~$g$, the global
infimum~$\sigma_{\mathcal{H}}^{\phantom{I}}(g)$ is attained by a
nonpositively curved piecewise flat metric on a genus~$g$ surface with
at most~$\mathcal{N}_0 \leq 2^{25} \, g^4 \log^2(g)$ conical singularities.
\end{corollary} 

This type of result is apparently of interest in closed string theory;
see the work by Zwiebach and coauthors \cite{WZ}, \cite{Za91},
\cite{He18a}, \cite{He18b}, \cite{Na19}.  \\

Systolically extremal surfaces without curvature assumptions, if they
are sufficiently regular (for instance, if they have bounded integral
curvature in Alexandrov's sense; see~\cite{AZ}, \cite{res93},
\cite{tro}), are covered by their systolic loops.  This may no longer
be the case for locally extremal nonpositively curved surfaces.
Still, by analyzing the geometry and shape of nonsystolic domains of
locally extremal nonpositively curved surfaces in
Section~\ref{sec:shape}, we obtain the following.

\begin{corollary} 
\label{coro:connected}
The union of the systolic loops of a locally extremal nonpositively
curved surface is path-connected.
\end{corollary}

\begin{remark}
A important tool in our study is the kite excision trick which merges
pairs of conical singularities, while keeping the systole fixed and
strictly decreasing the area.  More precisely, this trick consists in
excising a flat kite from a surface and identifying pairs of adjacent
sides of the excised surface; see~Section~\ref{six} for details.
\end{remark}

\subsection{Some open problems}

We conclude this introduction with a few open problems of varying
levels of difficulty.
\begin{enumerate}
\item By Corollary~\ref{coro:connected}, the systolic region of a locally
extremal nonpositively curved surface is connected.
Does this result still hold for the interior of the systolic region?
\item 
As we show in this article, the set of locally extremal nonpositively
curved metrics on a given genus~$g$ surface lies in a
finite-dimensional space.  A natural question to ask is whether the
set of locally extremal surfaces is finite, as it is the case for
hyperbolic surfaces, see~\cite{bav97} for a general setting.  
\item
Continuing with the previous item: do isosystolic deformations (i.e.,
deformations preserving the systolic area) of locally extremal
nonpositively curved surfaces exist?
\item 
To what extent can one relax the nonpositive curvature condition?  For
instance, what can be said about extremal metrics of curvature at
most~$\varepsilon$ with unit systole for small~$\varepsilon>0$?
\end{enumerate}

\subsubsection*{Acknowledgements} 

The second author would like to thank Thomas Richard and Marc Troyanov
for interesting discussions about Alexandrov surfaces.  We are
grateful to Barton Zwiebach for helpful comments on an earlier version
of the manuscript.

\section{Strategy}

Let us comment on the strategy of the proof of
Theorem~\ref{theo:main}.  We first discuss the general idea on a
sufficiently smooth extremal surface pointing out the main
difficulties.  After presenting a natural attempt to overcome these
difficulties in the context of Alexandrov surfaces, we finally develop
the strategy of the proof, sketching the argument.

\subsection{Local perturbation of extremal surfaces}

Suppose first that an extremal metric with nonpositive curvature
exists on a given surface and that this metric is sufficiently smooth.
By the flat strip theorem (see \cite[\S II.2.13]{BH}), two homotopic
systolic loops on a nonpositively curved surface bound a flat annulus
foliated by systolic loops.  Away from these flat systolic bands, in
regions where no systolic loop passes, the extremal surface must be
flat, for otherwise its curvature would be negative and its area could
be decreased by a local perturbation of the conformal factor,
affecting neither the systole, nor the sign of curvature, and
contradicting the extremality of the surface.  Of course, this
argument only holds for regions where the metric is smooth enough.  In
particular, it does not shed any light on the nature of the
singularities of the extremal metric, which necessarily exist in genus
at least two, otherwise the extremal surface would be flat.  Thus,
though appealing, this argument does not prove anything if we cannot
establish the existence of a smooth enough extremal metric \emph{a
priori}, which amounts to a classical issue in the calculus of
variations.

The existence of extremal metrics in a given conformal class can be
derived from compactness results on the conformal factor using its
log-subharmonicity in nonpositive curvature.  However, the regularity
of the metrics thus obtained is too weak for our purposes.  Moreover,
it is unknown whether the systolically optimal metric in a given
conformal class has finitely many singularities or not.  The advantage
of our technique based on the kite excision move (see
Section~\ref{six}) is that it has the mobility of moving about freely
in the moduli space of conformal classes and is not constrained to a
single class.

\subsection{Alexandrov surfaces}From a different (more geometric) 
point of view, the theory of Alexandrov surfaces with bounded integral
curvature provides the desired features regarding curvature measure,
compactness results and conformal representation; see~\cite{AZ},
\cite{res93}, \cite{tro}.  Loosely speaking, every Alexandrov surface
with bounded integral curvature can be described by its conformal
structure, represented by a (hyperbolic) Riemannian metric~$h$ of
curvature~$K_h$, and a curvature measure
\[
d\omega = K_h \, dA_h + d\mu
\]
where~$\mu$ is a Radon measure of total mass zero.  Here, the
function~$u$ in the conformal factor~$e^{2u}$ of the surface satisfies
$\Delta_h u = \mu$ in the distribution sense.  Therefore, it is
determined by the inverse of the Laplacian on~$(M,h)$ given by the
Green function~$G$, namely
\begin{equation} \label{eq:u}
u(x) = \int_M G(x,y) \, d\mu.
\end{equation}
Arguing as before, we seek to show that the curvature measure vanishes
in a neighborhood of a point where no systolic loop passes, by a
perturbation of the conformal factor.  For this purpose, we consider
variations of the Radon measure~$\mu$ in a neighborhood of this point,
leaving us with the following problem: even though the support of the
measure variation is localized in this neighborhood, we have no
control on the support of the variation of the conformal factor given
by~\eqref{eq:u}.  This could affect the value of the systole and,
therefore, the validity of the argument.  If the conformal factor is
modified in a given region the curvature measure is affected only in
this region, but it is not clear how to read off this property from
the curvature measure variation.

\subsection{A priori bounds}

We will follow a different strategy enabling us to establish \emph{a
priori} upper bounds on the number of conical singularities.  The
argument proceeds as follows.  In Section~\ref{sec:compact}, we first
recall that every nonpositively curved surface can be approximated by
a nonpositively curved piecewise flat surface with conical
singularities.  We also show that the systolic area defines a proper
functional when restricted to the moduli space of nonpositively curved
piecewise flat metrics whose number of conical singularity is
uniformly bounded.  This compactness result will allow us to derive
the existence of locally extremal surfaces from an \emph{a priori} upper bound
on the number of conical singularities of an almost locally extremal piecewise
flat surface.  Such an upper bound follows from a polynomial bound on
the number of systolic loops up to homotopy; see~Section~\ref{two}.

More specifically, flat systolic bands and isolated systolic loops
decompose any nonpositively curved piecewise flat surface into
nonsystolic polygonal regions whose number of edges is related to the
number of systolic homotopy classes, and is therefore uniformly
bounded; see~Section~\ref{sec:decomp}.  We introduce the kite excision
trick in Section~\ref{six}.  We exploit the trick to deduce that each
nonsystolic polygonal region of an almost locally extremal piecewise flat
surface has at most one conical singularity;
see~Sections~\ref{sec:exploit} and~\ref{sec:shape}.  It follows that
the number of conical singularities of this surface is uniformly
bounded as desired.  

The kite excision trick has the effect of moving a pair of
singularities lying in the same nonsystolic region closer and closer
until they merge into a single nonpositively curved conical
singularity, while keeping the systole fixed and strictly decreasing
the area.
More precisely, it consists in excising a flat kite from a surface and identifying pairs of adjacent sides of the excised surface; see~Sections~\ref{six} and~\ref{sec:comparison}.  Moreover,
this construction gives a way of reaching a locally extremal surface. \\

We will assume throughout that all surfaces are of genus at least~$2$ to
avoid the torus case where the extremal systolic problem was
completely solved by Loewner; see~\cite[Theorem~5.4.1]{katz}.

\section{Metric approximation and compactness} \label{sec:compact}

We present a few classical results which will be used in the proof of
the existence of locally extremal nonpositively curved piecewise flat metrics
for each genus~$g$; see~Theorem~\ref{theo:bound}.  We start with the
following metric approximation result.

\begin{proposition} 
\label{prop:toponogov}
For every genus~$g\geq2$, the infimum of the systolic area over the
following three spaces yields the same
value~$\sigma_{\mathcal{H}}^{\phantom{I}}(g)$:
\begin{enumerate}
\item
over all nonpositively curved Riemannian metrics;
\item
over all nonpositively curved Riemannian metrics with conical
singularities in the sense of Definition~$\ref{d11}$;
\item
over all nonpositively curved piecewise flat metrics with conical
singularities.
\end{enumerate}
\end{proposition}

\begin{proof}
A metric~$M \in \mathcal{H}_g$ (\eg, a nonpositively
curved metric with conical singularities) can be approximated by a
smooth one of nonpositive curvature by smoothing out each of the
conical singularities, without significantly affecting the area and
the systole.

Next, rescale the smooth surface~$M$ (without changing the systolic
area) so that its Gaussian curvature~$K$ satisfies~$-1\leq K \leq 0$.
For every~$\varepsilon>0$, we can partition~$M$ into sufficiently
small right-angled geodesic triangles~$\Delta\subseteq M$ so
that~$\area(\Delta) \geq (1-\varepsilon) \,\area(\Delta_0)$,
where~$\Delta_0$ is the corresponding flat triangle with the same
sidelengths.  Indeed, by the Alexandrov--Toponogov comparison theorem,
comparing~$M$ with the spaceform of \emph{smaller} constant curvature
(namely,~$-1$), the area of~$\Delta$ is at least the area of the
comparison right-angled hyperbolic triangle.  The area of the
hyperbolic triangle is~$2\arctan(\tanh \frac{a}{2}\tanh\frac{b}{2})$
where~$a,b$ are the two sides.  Thus the lower bound
on~$\area(\Delta)$ can be made as close to~$\frac12ab$ as we wish
for~$a,b$ small enough, exploiting the developments of~$\arctan$
and~$\tanh$.

By the Alexandrov--Toponogov comparison theorem, comparing with the
spaceform of \emph{greater} constant curvature (namely,~$0$), the
angles of~$\Delta_0$ are no smaller than the corresponding angles
of~$\Delta$.  Hence the total angles of the conical singularities of
the piecewise flat surface~$M_0$ obtained from~$M$ by replacing each
$\Delta$ by~$\Delta_0$ are at least~$2\pi$, so that~$M_0$ has
nonpositive curvature.

Since~$\area(\Delta_0)\leq\frac12ab$, replacing~$\Delta$ by~$\Delta_0$
increases the area by a factor at most~$1+\varepsilon$, so that we
have tight control on the area of~$M_0$.  Meanwhile, each loop
in~$M_0$ decomposes into paths where each path is contained in a
suitable triangle~$\Delta_0\subseteq M_0$ with endpoints on the
boundary of~$\Delta_0$.  The corresponding path in~$\Delta\subseteq M$
is necessarily shorter by the Alexandrov--Toponogov comparison of~$M$
with the spaceform of \emph{greater} constant curvature (namely,~$0$)
and therefore~$\sys(M_0)\geq\sys(M)$ and
thus~$\sigma(M)\geq(1-\varepsilon)\,\sigma(M_0)$.
\end{proof}

More generally, one has the following result on metric approximation,
announced by Reshetnyak~\cite{Res59} and proved by
Yu.~Burago~\cite[Lemma~6]{Bur04}, in the more general setting of
Alexandrov surfaces.

\begin{proposition} 
\label{prop:approx}
Let~$M$ be a surface of genus~$g$.  Every Riemannian metric with
conical singularities on~$M$ is bilipschitz close to a piecewise flat
metric with conical singularities.  In particular, the systole and the
area of the two metrics are close.
\end{proposition}

\begin{proof}
The argument of~\cite[Lemma~6]{Bur04} proceeds as follows.  Construct
a suitable partition~$\mathcal{T}$ of~$M$ into small geodesic
triangles, where the conical singularities of~$M$ are located at the
vertices and where each triangle of~$\mathcal{T}$ is bilipschitz close
to its comparison flat triangle with the same side lengths.  Replacing
each triangle of~$\mathcal{T}$ with its comparison flat triangle gives
rise to a piecewise flat metric with conical singularities on~$M$.
Putting together the bilipschitz maps between triangles yields a
bilipschitz map between the two metrics on~$M$, with bilipschitz
constant close to~$1$.
\end{proof}

\begin{corollary} \label{coro:approx}
In Proposition~$\ref{prop:approx}$, if the
Riemannian metric with conical singularities on~$M$ is nonpositively
curved then the piecewise flat metric can be assumed to be
nonpositively curved, as well.
\end{corollary}

\begin{proof}
Examining the construction in the proof of
Proposition~$\ref{prop:approx}$, we note that if the initial
Riemannian metric with conical singularities on~$M$ is nonpositively
curved, then an Alexandrov--Toponogov comparison of triangle angles
shows that the associated piecewise flat metric has nonpositive
curvature as at the end of the proof of
Proposition~\ref{prop:toponogov}.
\end{proof}

\begin{remark}
As mentioned earlier, the metric constructions in the proofs of
Proposition~\ref{prop:approx} and Corollary~\ref{coro:approx} go
through in the class of Alexandrov surfaces.  Thus, these two results
still hold for Alexandrov surfaces.  In particular, the optimal
systolic area~$\sigma_{\mathcal{H}}^{\phantom{I}}(g)$ for nonpositive
curvature agrees with the infimum of the systolic area over all
genus~$g$ surfaces with nonpositive curvature in the sense of
Alexandrov, or equivalently CAT(0) surfaces.
\end{remark}

The next result on conformal compactness
(Proposition~\ref{prop:conformal}) is due to Gromov~\cite[\S 5]{Gr83}.
We present the proof since the original arguments are somewhat
scattered.  We first state a result used in the proof.

\begin{theorem}[The collar theorem] 
\label{t34}
On a closed hyperbolic surface~$M$, a simple closed geodesic~$\gamma$
of hyperbolic length~$\ell$ admits a tubular neighborhood
\[
\mathcal{C} = \{ x \in M \mid d_{\rm hyp}(x,\gamma) < w \}
\]
of width
\begin{equation}
\label{e31}
w=\arsinh \left( \frac{1}{\sinh(\frac{\ell}{2})} \right)
\end{equation}
diffeomorphic to an annulus.
\end{theorem}

A proof can be found in~\cite[Theorem~4.1.1]{Bus92}.

\begin{proposition} 
\label{prop:conformal}
Let~$M$ be a closed surface of genus~$g \geq 2$, and let~$K>0$.  The
space of conformal classes of Riemannian metrics (possibly with
conical singularities) on~$M$ with systolic area at most~$K$ is a
compact set in the conformal moduli space~$\mathcal{M}_g$.
\end{proposition}

\begin{proof}
The capacity of an annulus~$\mathcal{C}$ endowed with a Riemannian
metric possibly with conical singularities is a conformal invariant
defined as
\[
\Capacity(\mathcal{C}) = \inf_F \int_\mathcal{C} |dF|^2 \, dA
\]
where~$F$ is a smooth function on~$\mathcal{C}$ equal to~$0$ on one
boundary component of~$\mathcal{C}$ and to~$1$ on the other.  The
capacity of an annulus~$\mathcal{C}\subseteq M$ around a
noncontractible simple closed curve of~$M$ is related to the systolic
area of~$M$ by the inequality
\begin{equation} 
\label{eq:cap}
\Capacity(\mathcal{C}) \geq \frac{1}{\sigma(M)};
\end{equation}
see~\cite[\S 5]{Gr83} and~\cite[\S 2.D.6]{Gr96} for further detail.

Let~$\gamma$ be a systolic loop of length~$\ell$ for the hyperbolic
metric conformally equivalent to the metric~$M$.  By
Buser--Sarnak~\cite[(3.4)]{BS94}, the capacity of the collar provided
by Theorem~\ref{t34} is given by
\[
\Capacity(\mathcal{C}) = \frac{\ell}{\pi - 2 \theta_0}
\]
where
\[
\theta_0 = \arcsin \left( \frac{1}{\cosh(w)} \right)
\]
with~$w$ as in \eqref{e31}.  Therefore, the capacity of the
annulus~$\mathcal{C}$ tends to zero with the (hyperbolic) length
of~$\gamma$.

Since the capacity is a conformal invariant, we deduce
from~\eqref{eq:cap} and the assumption on the systolic area of~$M$,
that the systole~$\ell$ of the hyperbolic metric conformally
equivalent to~$M$ is bounded away from zero.  By Mumford~\cite{Mum71},
this implies that the conformal hyperbolic metric, and so the
conformal class of~$M$, varies through a compact set in the conformal
moduli space~$\mathcal{M}_g$.
\end{proof}

As a consequence of the previous result, we obtain that the systolic
area function is proper on a suitable moduli space, in the following
sense.

\begin{proposition} 
\label{prop:compact}
Let~$N$ be an arbitrary natural number.  The space of piecewise flat
nonpositively curved metrics with at most~$N$ conical singularities on
a closed surface~$M$ of genus~$g \geq 2$ of systole normalized to~$1$
and area bounded above is compact.
\end{proposition}

\begin{proof}
By Proposition~\ref{prop:conformal}, the conformal classes of the
metrics considered in Proposition~\ref{prop:compact} lie in a compact
set~$\mathcal{K}\subseteq\mathcal{M}_g$.

By Troyanov~\cite[\S 5]{tro86}, each conformal class of~$M$ carries a
piecewise flat conformal metric with at most~$N$ prescribed conical
singularities~$p_i$ of given total angles~$\theta_i$, provided that
the Gauss--Bonnet relation
\begin{equation} 
\label{eq:GBN}
\sum_{i=1}^N (\theta_i - 2\pi) = 4\pi (g-1).
\end{equation}
(see~\cite[\S3]{tro86}) is satisfied.  This metric is unique upon
normalization to unit systole.  Furthermore, the dependence on
parameters is continuous (see also~\cite{Tr07}).

As the metric is nonpositively curved, the angles~$\theta_i$ are at
least~$2 \pi$ and so lie in the interval~$[2\pi,(4g-2)\pi]$.  Since
the Gauss--Bonnet relation~\eqref{eq:GBN} is closed,
the~$N$-tuple~$(\theta_1,\cdots,\theta_N)$ ranges through a compact
set~$L\subseteq\R^N$.  Thus, the space of piecewise flat nonpositively
curved metrics with at most~$N$ conical singularities on a closed
genus~$g$ surface of systole normalized to~$1$ and area bounded above
is homeomorphic to a compact subset of~$\mathcal{K} \times M^N \times
L$.
\end{proof}

\section{Systolic bands}
\label{two}

The goal of this section is to present some geometric properties
related to the notion of systolic bands, based of the flat strip
theorem and Greene's results \cite{G}.

\medskip

The following result is immediate from the flat strip theorem; see
Bridson--Haefliger \cite[\S II.2.13]{BH}.

\begin{lemma}
\label{lem:annulus}
Let~$M$ be a closed nonpositively curved surface with finitely many
conical singularities.  Then every pair of homotopic simple closed
geodesics bounds a flat annulus in~$M$.
\end{lemma}

\begin{definition}
\label{r42}
A curve passing through a singular~$p\in M$ splits the total angle
$\theta_p>2\pi$ into two \emph{rotation angles}~$R_p$ and~$L_p$ with
\mbox{$R_p+L_p=\theta_p$}.
\end{definition}

The local condition defining a geodesic encountering a singular point
requires that the rotation angles satisfy~$R_p\geq\pi$
and~$L_p\geq\pi$.

\begin{definition} 
\label{def:bands}
Let~$M$ be a closed nonpositively curved surface with finitely many
conical singularities.  A \emph{systolic homotopy class} of~$M$ is a
free homotopy class of loops containing a systolic loop.  Every
systolic homotopy class~$\mathcal{C}$ of~$M$ gives rise to a closed
\emph{systolic band} in~$M$ defined as the union of the systolic loops
in~$\mathcal{C}$.
\end{definition}

By Lemma~\ref{lem:annulus}, there are two possibilities for a systolic
band.

\begin{definition}
\label{def:bands2}
A systolic band is formed of
\begin{enumerate}
\item 
either an \emph{isolated systolic loop} when there is only one
systolic loop in the corresponding systolic homotopy class; or
\item 
\label{two2}
a flat open annulus bounded by two \emph{limit systolic loops} (which
are not necessarily disjoint) and foliated by systolic loops.%
\footnote{The case of the torus is exceptional and was already
excluded from the outset.}
\end{enumerate}
In case~\eqref{two2}, we refer to the systolic band as a \emph{fat
systolic band}.
\end{definition}

Observe that a closed fat systolic band is not always homeomorphic to
a closed annulus as its two limit systolic loops are not necessarily
disjoint. \\

We will need the following result of Greene~\cite{G}.

\begin{theorem} 
[Greene]
\label{theo:ABG}
Let~$M$ be a closed surface of genus~\mbox{$g \geq 2$}.  Then the
number of pairwise nonhomotopic simple closed curves on~$M$ meeting
each other at most once is bounded by
\begin{equation} \label{eq:G}
Q(g) \leq 2^9 \, g^2 \log g.
\end{equation}
\end{theorem}

The multiplicative constant~$2^9$ does not appear in \cite{G}.
However, going through the argument of~\cite{G}, we obtain the bound
$8x \log_2 x$ for~$Q(g)$, where~$x=8(g-1)(2g-1)$, which leads to the
multiplicative constant~$2^9$ in~\eqref{eq:G}. \\

Now, we are ready to prove the following proposition.

\begin{proposition} 
\label{prop:kissing}
Let~$M$ be a closed nonpositively curved surface of genus~\mbox{$g
\geq 2$} with finitely many conical singularities.  Then
\begin{enumerate}
\item 
Each pair of intersecting systolic loops of~$M$ meet exactly at one or
two points, or along an arc;\label{item:kissing1}
\item
When two systolic loops meet at two points exactly, this pair of
points decomposes each of the two systolic loops into geodesic arcs of
the same length.
\label{item:kissing2}
\end{enumerate}
\end{proposition}

\begin{proof}
To show~\eqref{item:kissing1} and~\eqref{item:kissing2}, let~$\alpha$ and~$\beta$ be two
systolic loops that meet in a single connected component, then they
meet either in one point or along an arc.

Suppose now that their intersection~$\alpha\cap\beta$ has at least two
connected components.  Then there exist two
subarcs~$\beta_1\subseteq\beta$ and~$\beta_2\subseteq\beta$ (in the
complement of~$\alpha\cap\beta$) with disjoint interior
meeting~$\alpha$ only at their endpoints.  The endpoints of the
arc~$\beta_i$ decompose~$\alpha$ into two arcs denoted~$\alpha'_i$
and~$\alpha''_i$.  Observe that none of the four loops~$\alpha'_1 \cup
\beta_1$,~$\alpha''_1 \cup \beta_1$,~$\alpha'_2 \cup \beta_2$
and~$\alpha''_2 \cup \beta_2$ is contractible, otherwise two distinct
geodesic arcs with the same endpoints would be homotopic, which is
impossible on a nonpositively curved surface.  Thus, each of these
four loops is of length at least~$\sys(M)$.  The sum of their lengths
is at least~$4 \, \sys(M)$ and at most twice the total length
of~$\alpha$ and~$\beta$:
\begin{equation}
\label{e21}
4\,\sys(M)\leq |\alpha'_1| + |\alpha''_1| + |\alpha'_2| + |\alpha''_2|
+2|\beta_1|+2|\beta_2| \leq2|\alpha|+2|\beta|.
\end{equation}
Hence both inequalities in \eqref{e21} are equalities and the same
holds for the four inequalities involved in the sum.  It follows that
each of the arcs~$\alpha_i'$,~$\alpha_i''$ and~$\beta_i$ is of
length~$\frac{1}{2} \sys(M)$.  Therefore, the only intersection points
between~$\alpha$ and~$\beta$ are the two endpoints~$p$ and~$q$ which
are necessarily antipodal points of both of these loops.
\end{proof}

\begin{proposition}
\label{item:kissing3}
Let~$M$ be a closed nonpositively curved surface of genus~\mbox{$g
\geq 2$} with finitely many conical singularities.  Then the number
$\bar Q(g)$ of systolic homotopy classes is at most
\[
\bar{Q}(g) \leq 32(g-1)^2 + Q(g) \leq 2^{10} \, g^2 \log g.
\]
\end{proposition}

\begin{proof}
Let us estimate first the number of homotopy classes of systolic loops
meeting at exactly two points.  Let~$\alpha, \beta$ be two such loops.
The case of equality in inequality~\eqref{e21} implies that the
systolic loops~$\alpha$ and~$\beta$ meeting at points~$p$ and~$q$
decompose into four distinct length-minimizing arcs of
length~$\frac{1}{2} \sys(M)$ joining~$p$ and~$q$.

Let~$a$ and~$b$ be two length-minimizing arcs of length~$\frac{1}{2}
\sys(M)$ joining~$p$ and~$q$.  As~$M$ is nonpositively curved, the two
geodesic arcs~$a$ and~$b$ are nonhomotopic and form a systolic loop.
By first variation, the angle at~$p$ between the arcs~$a$ and~$b$ is
at least~$\pi$.  It follows that~$p$ is a conical singularity of total
angle~$\theta_p\geq4\pi$ and similarly for~$q$.

The lower bound on the angles between the length-minimizing arcs
joining~$p$ to~$q$ imply that there exist at
most~$\left\lfloor{\theta_p}/{\pi}\right\rfloor$ such
length-minimizing arcs.  This shows that there are at most
\begin{equation} \label{eq:bound}
\sum_{\theta_p \geq 4\pi} {\left\lfloor{\theta_p}/{\pi}\right\rfloor
\choose 2} \leq \frac{1}{2\pi^2} \sum_{\theta \geq 4\pi} \theta_p^2
\leq \frac{1}{2\pi^2} \left( \sum_{\theta \geq 4\pi} \theta_p
\right)^2
\end{equation}
systolic loops~$\alpha\subseteq M$ meeting another systolic
loop~$\beta$ at exactly two points.  the Gauss--Bonnet
formula~\eqref{eq:GBN} implies that whenever~\mbox{$\theta_p \geq
4\pi$}, we have~$\theta_p \leq 2(\theta_p -2\pi)$.  Since~$\theta_i >
2\pi$, it follows that
\[
\sum_{\theta_p \geq 4 \pi} \theta_p \leq 2 \sum_{\theta_p \geq 4 \pi} (\theta_p-2\pi) \leq 2 \sum_{i=1}^k (\theta_i - 2\pi) = 8\pi (g-1).
\]
We derive from~\eqref{eq:bound} that the number of systolic loops
intersecting another systolic loop at exactly two points is bounded by
\[
\frac{1}{2 \pi^2} \, [8\pi(g-1)]^2 = 32 (g-1)^2
\]
and so is at most quadratic.

It remains to estimate the number of homotopy classes of systolic
loops that do not meet any other loop in more than one point.  We
choose representative loops from these remaining classes, and deform
them so that each pair of loops meet at most at a single point.
By Theorem~\ref{theo:ABG}, the number of
pairwise nonhomotopic loops intersecting each other at most once is
bounded by~$Q(g) \leq 2^{9} \, g^2 \log g$.  
\end{proof}

\begin{remark}
Theorem~4 in~\cite{G} also provides an upper bound on the number of
pairwise nonhomotopic loops intersecting at most twice.  Directly
applying this result would yield an~$O \left( g^5 \log g \right)$
upper bound on the number of systolic homotopy classes in~$M$.  We
obtained a better almost quadratic bound by analyzing the special
structure of systolic loops intersecting twice, combined with the
almost quadratic bound of Theorem~\ref{theo:ABG} on the number of
pairwise nonhomotopic loops meeting at most once.  See
Przytycki\;\cite{Pr15} and Aougab-Biringer-Gaster\;\cite{ABG} for
earlier polynomial bounds, and Juvan-Malni\v{c}-Mohar\;\cite{JMM96} or
Malestein-Rivin-Theran\;\cite{MRT14} for even earlier exponential
ones.  Greene's almost quadratic upper bound~$O \left( g^2 \log g
\right)$ for pairwise nonhomotopic loops intersecting at most once can
be improved in the hyperbolic case to a subquadratic one~$O \left(
\frac{g^2}{\log g} \right)$ due to Parlier~\cite{Pa13}.  A
\emph{lower} bound of type~$O\left(g^{\frac43-\epsilon}\right)$ is due
to Schmutz-Schaller \cite{Sc98}.
\end{remark}

\begin{example} 
\label{ex:twice}
On smooth surfaces curves can be shortened by smoothing them out by
first variation, implying that systolic loops meet each other at most
once.  However, on a singular surface they may intersect twice, even
in nonpositive curvature.  For example, consider the standard sphere
along with four meridians joining the two poles.  Replace each of the
four lune-shaped spherical regions bounded by the meridians by a flat
cylinder of circumference~$\pi$ and altitude at least~$\frac{\pi}{4}$,
where the bottom of each cylinder is glued in isometrically along the
boundary of each of the lunes.  The resulting surface~$X$ is a flat
four-holed sphere with two conical singularities (at former poles) of
total angle~$4\pi$, which can be turned into a nonpositively curved
piecewise flat genus~$3$ surface~$M= X \cup_{\partial X} X$ by gluing
another copy to it.  The surface~$M$ obtained in this way has pairs of
systolic loops meeting transversely twice.
\end{example}

\section{Systolic decomposition}
\label{sec:decomp}

In this section, we describe the systolic decomposition of a closed
nonpositively curved piecewise flat surface~$M$ of genus~$g$.  By the
curvature condition, each total angle~$\theta_i$ at a conical
singularity is greater than~$2\pi$.

\begin{definition}
A conical singularity~$p \in M$ is said to be \emph{large} if the
angle at~$p$ is at least~$3\pi$, and \emph{small} otherwise.
\end{definition}

\begin{lemma} 
\label{lem:large}
Let $M$ be a closed nonpositively curved piecewise flat surface.
There are at most~$4(g-1)$ large conical singularities on~$M$, each of
total angle at most~$2\pi (2g-1)$.
\end{lemma}

\begin{proof}
This is immediate from the Gauss--Bonnet formula~\eqref{eq:GBN} since
large conical singularities satisfy $\theta_i- 2 \pi \geq \pi$ and
small conical singularities satisfy $\theta_i - 2\pi > 0$ as $M$ is
nonpositively curved.
\end{proof}

We will need a few more definitions.

\begin{definition}
A conical singularity~$p \in M$ is \emph{special} if every point in a
neighborhood of~$p$ lies in a (fat) systolic band.  In more detail, a
special singularity lies on the boundary of several closed fat
systolic bands in such a way that the union of these bands contains an
open neighborhood of the singularity.
\end{definition}

\begin{definition}
The \emph{systolic decomposition} of~$M$ is a partition
\[
M=(\sqcup_i S_i) \sqcup (\sqcup_j D_j)
\]
of~$M$ into systolic domains~$S_i$ and
nonsystolic domains~$D_j$ where
\begin{enumerate}
\item 
each \emph{systolic domain}~$S_i$ is a connected component of the
union of the systolic bands of~$M$ (see~Definition~\ref{def:bands});
\item 
each \emph{nonsystolic domain}~$D_j$ is a connected component of the
complementary set in~$M$ of the systolic bands of~$M$.
\end{enumerate}
\end{definition}

The intersection pattern of the systolic bands of~$M$ described in
Proposition~\ref{prop:kissing}.\eqref{item:kissing1} shows that every
systolic and nonsystolic domain has a finite geodesic polygonal
structure described as follows.

\begin{definition} \label{def:vertices}
The \emph{vertices} of the polygonal structure are of two types:
\begin{enumerate}
\item
the intersection points between pairs of either isolated or limit
systolic loops when they meet at one or two points
(see~Definition~\ref{def:bands2});
\item
if systolic loops meet along a segment~$I\subseteq M$ (see
Proposition~\ref{prop:kissing}.\eqref{item:kissing1}) then the
endpoints of~$I$ (which are also conical singularities) are also taken
to be vertices.
\end{enumerate}
The \emph{edges} of a systolic or nonsystolic domain~$D$ are the
connected components of~$\partial D$ minus the vertices
of~$\partial{}D$.
\end{definition}

\begin{remark}
The vertices of a systolic or nonsystolic domain~$D$ are not
necessarily located at conical singularities and conical singularities
may lie in the interior of the edges of~$D$.
\end{remark}

We now describe the structure of the systolic decomposition of~$M$.
Recall that~$\bar Q$ is the maximal number of systolic homotopy
classes; see Proposition \ref{item:kissing3}.

\begin{proposition} 
\label{prop:nonsystolic domains}
Let~$M$ be a piecewise flat nonpositively curved surface of genus~$g \geq 2$.
Let~$\mathcal{N} = 4 \, \bar{Q}(g)^2 \leq 2^{22} \, g^4 \log^2(g)$.  Then 
\begin{enumerate}
\item the corners of every nonsystolic domain at nonsingular points are convex, \ie, their angles are at most~$\pi$.
\item
the surface has at most~$\mathcal{N}$ special conical
singularities;
\item
the surface decomposes into at most~$\mathcal{N}$ nonsystolic domains;
\item
the surface has a total of at most~$\mathcal{N}$ edges.
\end{enumerate}
\end{proposition}

\begin{proof}
Let~$D$ be a nonsystolic domain of~$M$.  By definition (see
Definition~\ref{def:vertices}), the corners of~$D$ and the special
conical singularities%
\footnote{Though not required for our argument, note that the special
conical singularities of~$M$ can be thought of as degenerate
nonsystolic domains.}
of~$M$ correspond to intersections either
\begin{enumerate}
\item
between two systolic bands, giving rise to at most eight corners; or
\item 
between two \emph{homotopic} limit systolic loops meeting at one or
two points or along an arc, giving rise to at most four corners.
\end{enumerate}
Since there are at most~$\bar{Q}(g) \leq 2^{10} \, g^2 \log g$ systolic bands by~Proposition~\ref{item:kissing3}, this yields
at most
\[
\mathcal{N} =
8{{\bar{Q}(g)}\choose2}+4\,\bar{Q}(g) = 4 \, \bar{Q}(g)^2  \leq2^{22}\, g^4 \log^2(g)
\] 
corners and
special conical singularities.  In particular, the surface~$M$
decomposes into at most~$\mathcal{N}$ nonsystolic domains with a total
number of at most~$\mathcal{N}$ edges.
\end{proof}

\section{The kite excision trick}
\label{six}

In this section, we describe the kite excision trick, a key tool in
the proof of Theorem~\ref{theo:main}.  Consider a nonpositively curved
piecewise flat surface~$M$.  Let~$p,q\in M$ be two conical
singularities connected by a geodesic arc~$[p,q]$ with no conical
singularity lying in the interior~$(p,q)$.  Denote by~$\theta_p$
and~$\theta_q$ the total angles at~$p$ and~$q$.

\begin{definition}[Kite] 
\label{def:kite}
Let~$r\in M$ (not on~$[p,q]$) be a point such that the triangle~$pqr$
is flat.  Conside the reflection~$pqr'$ of triangle~$pqr$ with respect
to~$[p,q]$.  Define the \emph{kite}~$K=prqr'$ as the union of the two
symmetric triangles; see~Figure~\ref{fig:kite1}.
The two opposite vertices~$p$ and~$q$ of~$K$ are referred to as the
\emph{main vertices} of the kite.  The \emph{width}~$w$ of~$K$ is the
length of the diagonal~$[r,r']\subseteq K$.
\end{definition}

\begin{definition}[Admissible kite]
The kite~$K$ is \emph{admissible} if all its angles are less
than~$\pi$ and its angles at the main vertices~$p$ and~$q$ are related
to the angle excesses of the conical singularities~$p$ and~$q$ as
follows:
\begin{align*}
\measuredangle rpr' & \leq \min\{ \theta_p-2\pi, \pi \}  \\
\measuredangle rqr' & \leq \min \{ \theta_q-2\pi, \pi \}.
\end{align*}
\end{definition}

\begin{figure}[htb] 
\vspace{0.2cm} \def\svgwidth{6cm} 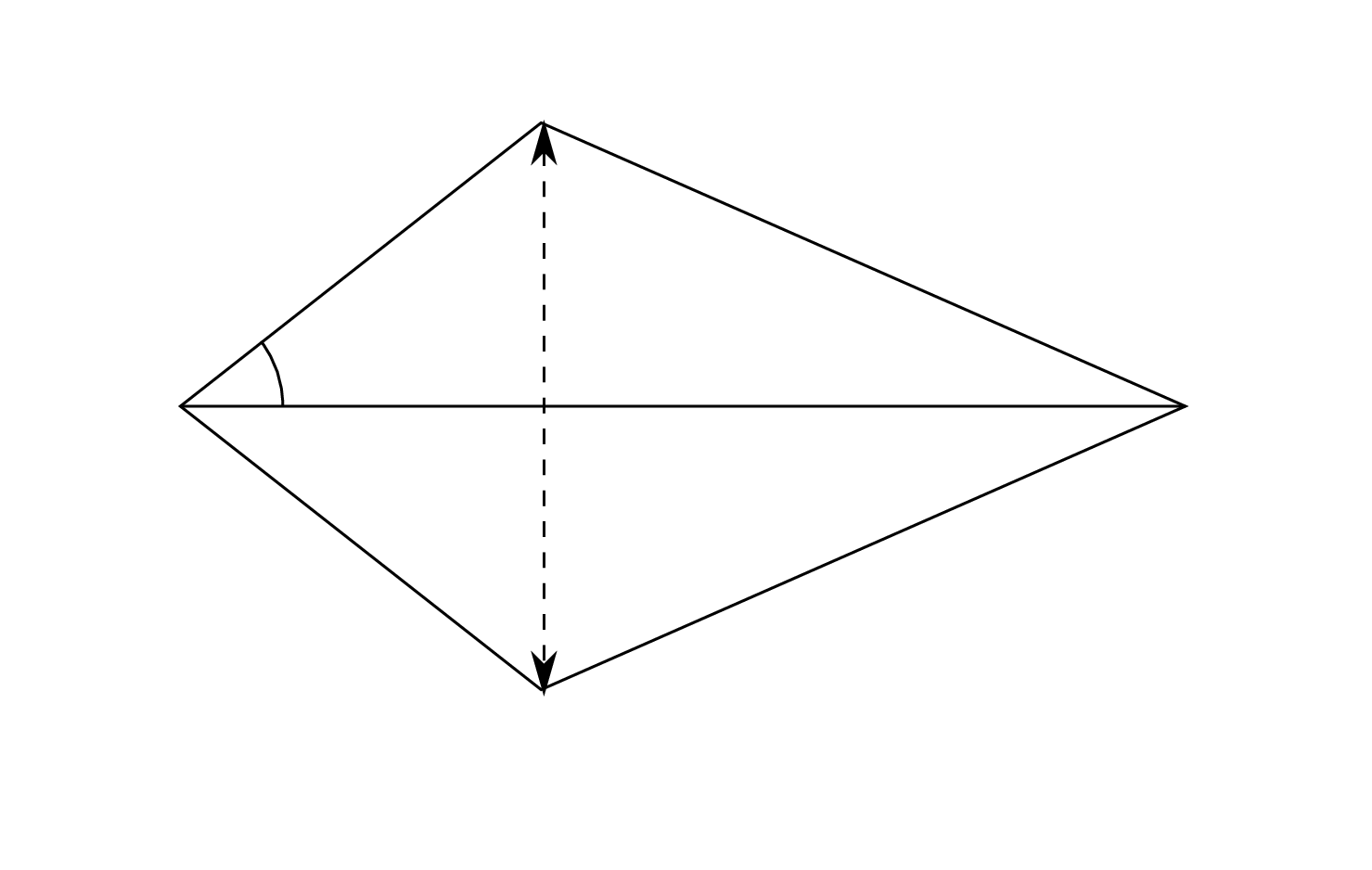
\vspace{0.2cm}\caption{The flat kite~$K$ of width~$w$}
\label{fig:kite1}
\end{figure}

\begin{definition}[Exact and diamond kites]
When~$p$ is a small conical singularity, an admissible kite~$K$ is
\emph{exact at}~$p$ if the following equality is satisfied:
\begin{align*}
\measuredangle rpr' & = \theta_p-2\pi < \pi. 
\end{align*}
$K$ is called a \emph{diamond kite} if~$|pr|=|qr|$.
\end{definition}

\begin{remark}
By construction, every admissible kite~$K$ is convex. 
\end{remark}

\begin{definition}[Excised surface~$M_w$] 
\label{def:excision}
Let~$K_w\subseteq M$ be an admissible kite of width~$w=|rr'|$.  We
perform a cut-and-paste procedure on~$M$ as follows.  We excise the
kite~$K_w\subseteq M$ and introduce identifications on the boundary
of~$M\setminus K$ by setting~$[p,r]\sim [p,r']$ and~$[q,r] \sim
[q,r']$.  The result of the surgery is a piecewise flat surface
\begin{equation}
\label{e61}
M_w=(M\setminus K_w)/\!\sim
\end{equation}
of genus~$g$ with conical singularities.
\end{definition}
Note that~$\area(M_w)<\area(M)$.
\begin{definition} 
\label{rem:quotient}
The quotient map
\begin{equation}
\label{e62b}
\pi_w\colon M \to M_w
\end{equation}
is obtained by collapsing each segment of~$K_w$ parallel to the
diagonal~$[r,r']$ to a point.
\end{definition}
The map~$\pi_w$ is a homotopy equivalence.

\begin{proposition}
\label{prop:same}
If~$K_w$ is an admissible kite then the excised surface~$M_w$ is
nonpositively curved with at most one more conical singularity
than~$M$.  Furthermore, if~$K_w$ is exact at one of its main vertices,
then the surface~$M_w$ has at most as many conical singularities
as~$M$.
\end{proposition}

\begin{proof}
The first statement follows by analyzing the total angles of the
points corresponding to the vertices of~$K_w$ and showing that they are
at least~$2 \pi$.  More precisely, the total angles at the points~$p$
and~$q$ in the excised surface~$M_w$ are~$\theta_p - \measuredangle
rpr'$ and~$\theta_q-\measuredangle rqr'$, both of which are at
least~$2 \pi$ since the kite is admissible.  Similarly, the total
angle at the point~$r=r'\in M_w$
is~$2\pi+\measuredangle{}rpr'+\measuredangle rqr'$.

For the second statement, even if the point~$r$ (identified with~$r'$)
creates a new conical singularity, the point~$p$ of total angle~$2\pi$
is no longer a singularity in the new surface~$M_w$.
\end{proof}

\begin{proposition}  
\label{prop:conv}
\label{prop:local}
Let~$p,q\in M$.  Consider an admissible kite~$K_w$ with main
diagonal~$[p,q]$ which is either a diamond or an exact kite at~$p$.
Then the excised surface~$M_w$ converges to~$M$, both for the
Gromov--Hausdorff distance and the Lipschitz distance, as the
width~$w$ of~$K_w$ tends to zero.
\end{proposition}

\begin{proof}
Fix an admissible diamond~$K_D = pr_0qr'_0$ with main
diagonal~$[p,q]$.  Consider a smaller admissible diamond~$K_w = prqr'$
of width~$w$, and build the excised surface~$M_w=(M\setminus
K_w)/\!\sim$ as in~\eqref{e61}.  Let~$s$ be the midpoint of~$[p,q]$,
so that~$r\in (r_0,s)$.  The diamond is the union of two triangles,
$pr_0r$ and~$qr_0r$.  We will need a map~$\phi_w$ defined as follows.

\begin{definition}
\label{d69}
Consider the linear map~$\phi_w\colon pr_0r\to pr_0s$ (respectively,
$\phi_w\colon qr_0r\to qr_0s$) fixing~$[p,r_0]$ and mapping the
triangle~$pr_0r$ (respectively,~$qr_0r$) to the right-angle
triangle~$pr_0s$ (respectively,~$qr_0s$).  We extend the linear map to
a continuous map
\begin{equation}
\label{e62}
\phi_w\colon M_w\to M
\end{equation}
by the identity map on the complement in~$M_w$ of~$K_D\setminus K_w$.
\end{definition}

The map~$\phi_w$ is clearly~$(1+\epsilon)$-bilipschitz (i.e., the
bilipschitz constant tends to~$1$ as~$w$ tends to zero).  It follows
that the quadrilateral~$pr_0qr$ is~$(1+\epsilon)$-bilipschitz with the
triangle~$pr_0q$.  By symmetry, the same holds with the
quadrilateral~$pr_0'qr'$ and the triangle~$pr_0'q$.  Thus the map
$\phi_w$ is~$(1+\epsilon)$-bilipschitz.  The surfaces are therefore
also Gromov--Hausdorff close.

Now consider the case of a kite~$K_E = pr_0qr'_0$ exact at~$p$.
Consider a point~$p_*$ close to~$p$ such that~$p_*$ is on a geodesic
extension~$p_*q$ of~$[p,q]$ so that the rotation angle of~$p_*q$
at~$p$ is equal to~$\frac{\theta_p}{2} \geq \pi$ on either side of the
segment~$p_*q$.  Since the kite~$K_E$ is exact at~$p$, we
have~$p\in[p_*,r_0]$.  Consider the segment~$[p_*,r_0]$
containing~$p$.  Fix a circular arc~$\wideparen{p_*r_0} \subseteq M
\setminus K_E$ bounding a flat region~$\mathcal{R}$ together with the
segment~$[p_*,r_0]$ containing~$p$.  Take a smaller kite~$K_w =
prqr'\subseteq K_E$ of width~$w$ and exact at~$p$,
where~$r\in(p,r_0)$.  Let~$p_r \in [p_*,p]$ with~$|pp_r| = |pr|$.
Note that the rotation angle of~$\mathcal R$ at~$p$ is precisely~$\pi$
by the exactness hypothesis.  The rotation angle is also~$\pi$ at~$r$
and~$p_r$ by construction.  There exists a~$(1+\epsilon)$-bilipschitz
homeomorphism
\begin{equation}
\label{e64}
h_r\colon \mathcal{R} \to \mathcal{R}
\end{equation}
which fixes the circular arc~$\wideparen{p_*r_0}$ pointwise and
linearly maps~$[r_0,r]$,~$[r,p]$, and~$[p,p_*]$
to~$[r_0,p]$,~$[p,p_r]$, and~$[p_r,p_*]$, respectively,
where~$\epsilon$ tends to~$0$ as~$r$ approaches~$p$.

\begin{definition}
\label{d610}
We combine the map~$h_r$ of \eqref{e64} with
the~\mbox{$(1+\epsilon)$}-bilipschitz linear map from~$rr_0q$
to~$pr_0q$ fixing~$[r_0,q]$, and perform a symmetric construction on
the other half of the kite, to produce a map
\begin{equation}
\label{e65}
\phi_w\colon M_w\to M \quad \text{where} \quad M_w=(M\setminus
K_E)/\!\sim.
\end{equation}
\end{definition}
The resulting map~$\phi_w\colon M_w\to M$ is
$(1+\varepsilon)$-bilipschitz, as in the diamond case.
\end{proof}

\section{Systole comparison} 
\label{sec:comparison}

Let~$M$ be a nonpositively curved piecewise flat surface of genus~$g$.
Consider a nonsystolic domain~$D\subseteq M$.  Let~$p$ and~$q$ be
conical singularities in the closure~$\overline D$ of~$D$, joined by a
geodesic arc~$[p,q]\subseteq\overline D$.  Note that the arc may start
and end at the same point~$p=q$ in the cases ($D_1$) and ($D''_1$)
below.  We can assume that no conical singularity lies in the
interior~$(p,q)$, by picking a different pair of conical singularities
along the arc, if necessary.  Recall that the set~$D\subseteq M$ is
open.

We will now choose an admissible kite~$K_w\subseteq M$ of width~$w$
constructed by symmetry with respect to~$[p,q]$
(see~Definition~\ref{def:kite}) in one of the following ways; see
Figures~\ref{fig:case1} through~\ref{fig:case4}.
\begin{enumerate}
\item[($D_1$)]
if~$[p,q]\subseteq D$, take a diamond~$K_w$ of sufficiently small
width so that it lies in~$D$;
\item[($D'_1$)] if~$[p,q)\subseteq D$ where~$q \in \partial D$ and the
angle of~$D$ at the point~$q$ is greater than~$\pi$, take a
diamond~$K_w$ of sufficiently small width so that~\mbox{$K_w \setminus
\{q\}$} lies in~$D$;
\item[($D''_1$)]
if~$(p,q)\subseteq D$ with~$p,q \in \partial D$ and the angles of~$D$
at~$p$ and~$q$ are greater than~$\pi$, take a diamond~$K_w$ of
sufficiently small width so that~$K_w \setminus \{p,q\}$ lies in~$D$;
\item[($E_1$)]
if~$[p,q]\subseteq D$ and~$p$ is a small conical singularity,
take~$K_w$ exact at~$p$ of sufficiently small width so that it lies
in~$D$;
\item[($E'_1$)]
if~$[p,q)\subseteq D$ with~$p$ a small conical singularity and~$q \in
\partial D$, and the angle of~$D$ at~$q$ is greater than~$\pi$,
take~$K_w$ exact at~$p$ of sufficiently small width so that~$K_w
\setminus \{q\}$ lies in~$D$;
\item[($E_2$)]
if~$[p,q]$ is contained in the interior of an edge of~$\partial D$
and~$p$ is a small conical singularity,
take~$K_w$ exact at~$p$ of sufficiently small width so that the part
of every systolic loop passing through~$K_w$ is parallel to~$[p,q]$.%
\footnote{Recall that, by definition of a nonsystolic domain~$D$, no
systolic loop meets the interior of an edge of~$\partial D$ unless it
contains this edge, which ensures the existence of such kites.}
\end{enumerate}

\begin{figure}[htb] 
    \centering
    \begin{minipage}{0.45\textwidth}
        \centering
        \def\svgwidth{4cm} 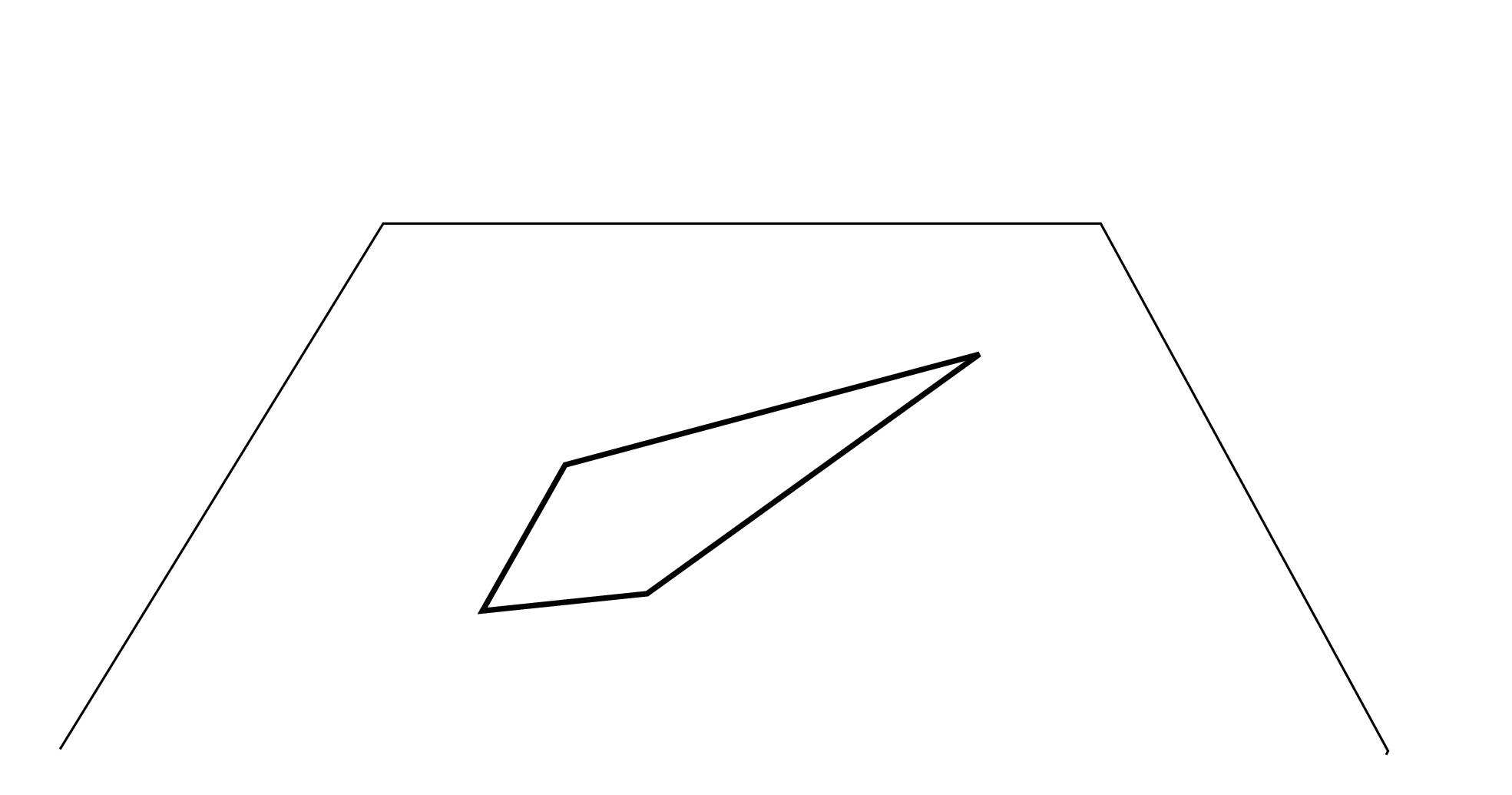 
                \caption{$(D_1), (E_1)$}
                \label{fig:case1}
    \end{minipage}\hfill
    \begin{minipage}{0.45\textwidth}
        \centering
        \def\svgwidth{4cm} 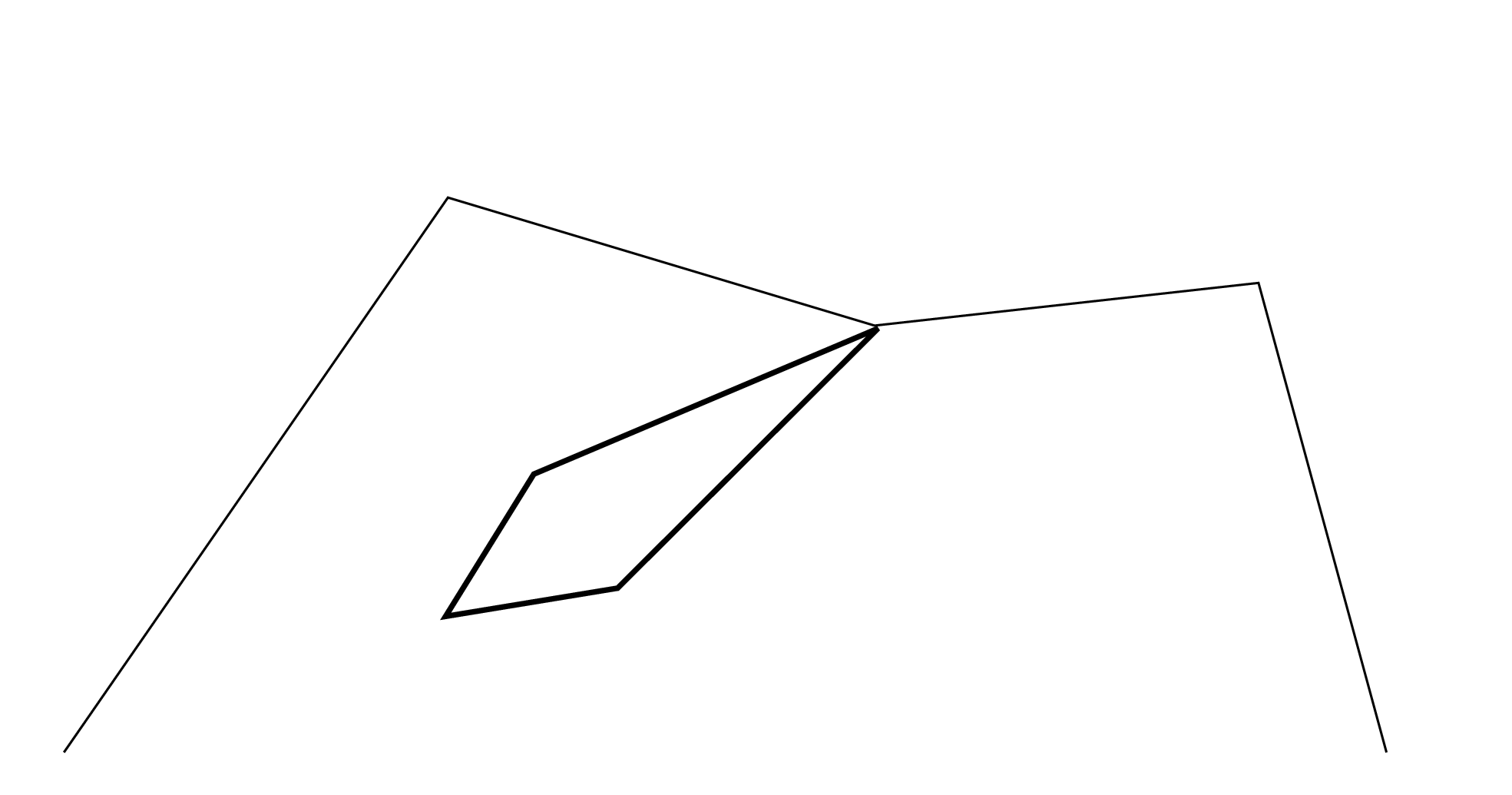 
        \caption{$(D'_1), (E'_1)$}
        \label{fig:case2}
    \end{minipage}\hfill
    \begin{minipage}{0.45\textwidth}
        \centering
        \def\svgwidth{4cm} 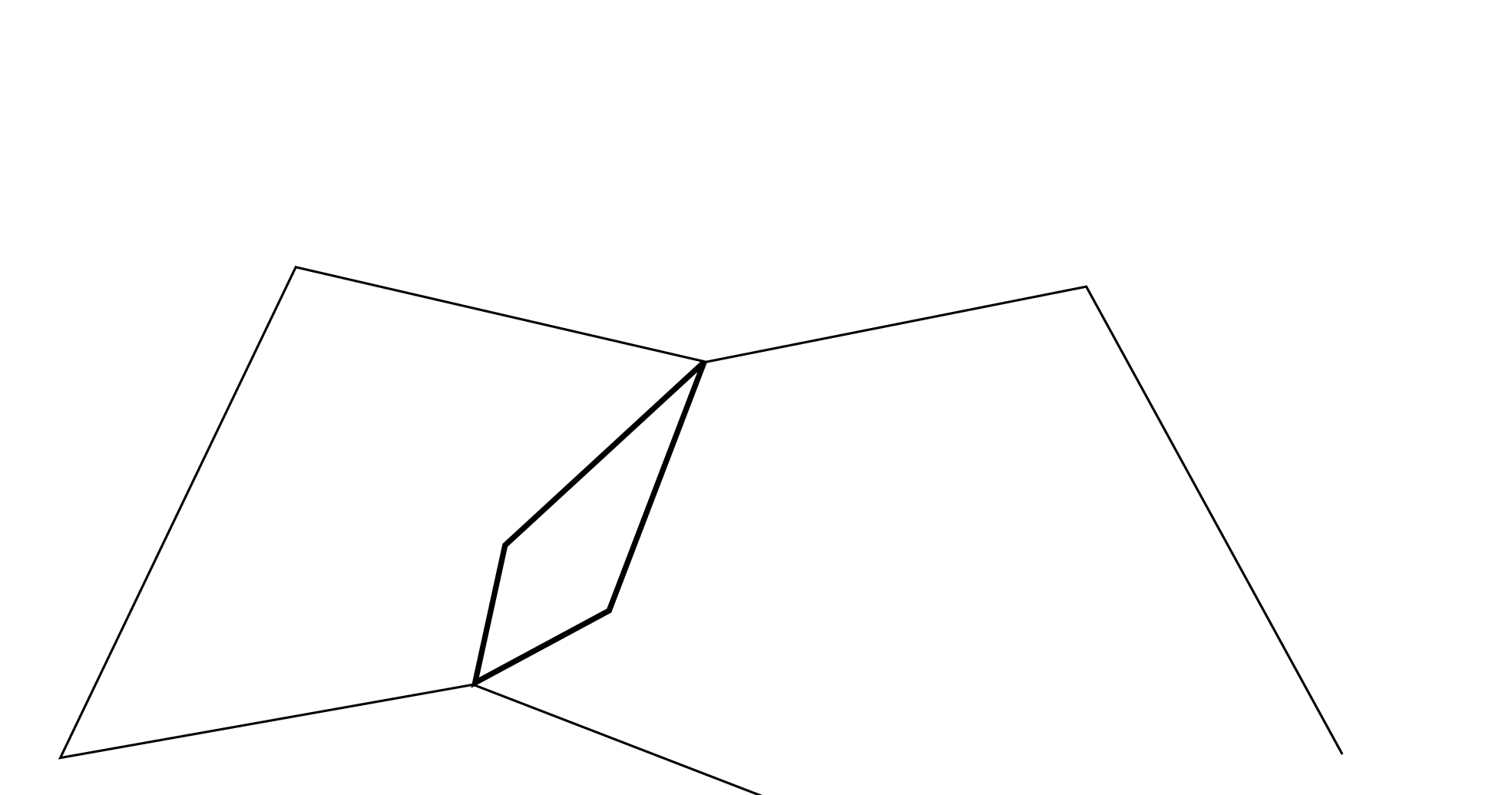  
        \caption{$(D''_1)$}
        \label{fig:case3}
    \end{minipage}\hfill
    \begin{minipage}{0.45\textwidth}
        \centering
        \def\svgwidth{4cm} 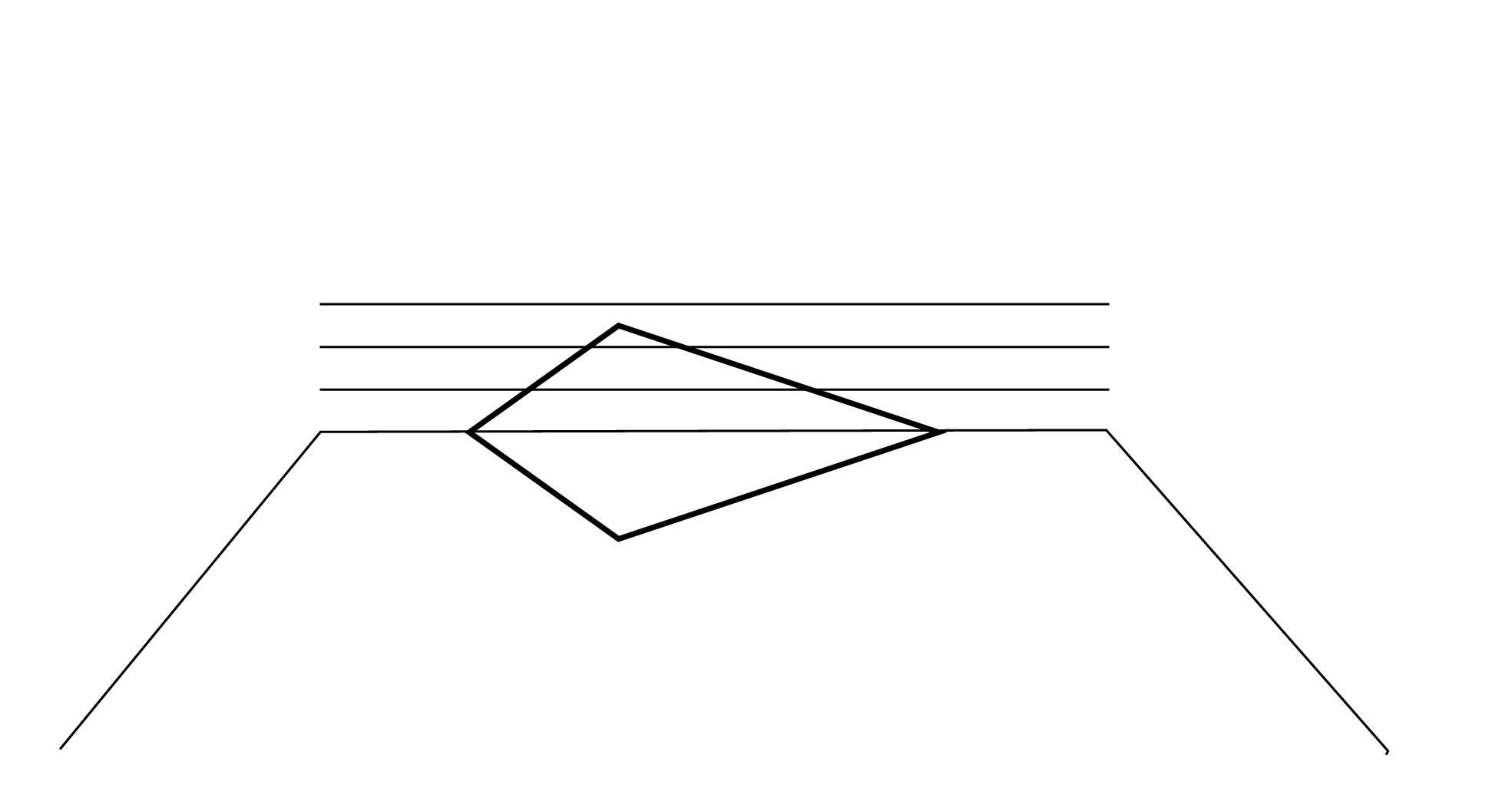  
        \caption{$(E_2)$}
        \label{fig:case4}
    \end{minipage}
\end{figure}

Consider the surface~$M_w=(M\setminus K_w)/\!\sim$\, where the
kite~$K_w$ is admissible and satisfies one of the previous hypotheses.
Only the properties for kites satisfying ($E_1$), ($E'_1$), and ($E_2$)
are required for the proof of our main result,
Theorem~\ref{theo:bound}.  In the proofs of
Proposition~\ref{prop:sys1} and Proposition~\ref{prop:sys2}, the
width~$w$ of~$K_w$ will need to be be chosen even smaller to satisfy
further restrictions.

Since the quotient map~$\pi_w\colon M \to M_w$ of \eqref{e62b} is a
nonexpanding homotopy equivalence, we have~$\sys(M_w) \leq \sys(M)$.
In the following, we will show that the reverse inequality holds in
the cases \textnormal{($D_1$)}, \textnormal{($E_1$)},
\textnormal{($D'_1$)}, \textnormal{($E'_1$)}, \textnormal{($D''_1$)}
and \textnormal{($E_2$)} as well.

\subsection{Analysis of cases ($D_1$) and ($E_1$)}

\begin{proposition} 
\label{prop:sys1}
Consider an admissible kite~$K_w\subseteq D$ satisfying
\textnormal{($D_1$)} or~\textnormal{($E_1$)}, so that~$K_w$ is either
an admissible diamond or an exact kite at~$p$.  If the width~$w$
of~$K_w$ is sufficiently small, then
\[
\sys(M_w) = \sys(M).
\]
\end{proposition}

\begin{proof}
We first consider a diamond kite~$K_D$ lying in a nonsystolic
domain~$D\subseteq M$ as in the proposition.  Consider the
function~$x\mapsto \sys(M,x)$ on~$M$, where~$\sys(M,x)$ represents the
least length of a shortest noncontractible loop based at the
point~\mbox{$x\in M$}.  Since~$K_D$ is a compact subset of a
nonsystolic domain, there exists an~$\varepsilon>0$ such that
\begin{equation} \label{e71}
\forall x \in K_D, \, \sys(M,x) > \sys(M) + \varepsilon.
\end{equation}
Consider a subkite~$K_w\subseteq K_D$.  We construct the
map~$\phi_w\colon M_w\to M$ using the pair~$K_w\subseteq K_D$ as in
Definition~\ref{d69}.  As~$w$ tends to~$0$, the bilipschitz constant
of~$\phi_w$ tends to~$1$.  Since the deformation~$M_w$ of~$M=M_0$ is
continuous with respect to the bilipschitz distance,
inequality~\eqref{e71} implies that for~$w_0$ sufficiently small, each
noncontractible loop~$\gamma\subseteq M_{w_0}$ based at a point
of~$\phi^{-1}_{w_0}(K_D)$ satisfies
\begin{equation}
\label{e72}
|\gamma|\geq\sys(M).
\end{equation}
Meanwhile if a loop~$\gamma\subseteq M_{w_0}$ is disjoint
from~$\phi^{-1}_{w_0}(K_D)$  then
\begin{equation}
\label{e73}
|\gamma|=|\phi_w^{-1}(\gamma)|\geq\sys(M)
\end{equation}
since~$\phi_{w_0}$ is an isometry outside of~$K_D$.  The
bounds~\eqref{e72} and\eqref{e73} prove the proposition in the case of
a diamond kite.

For an exact kite~$K_E$, we follow a similar procedure with~$\phi_w$
of Definition~\ref{d610}.
\end{proof}

\subsection{Analysis of cases ($D'_1$), ($D''_1$), and ($E'_1$)}

\begin{proposition} 
\label{prop:sys1b}
Consider an admissible kite~$K_w$ with main diagonal~$[p,q]$ with~$K_w
\setminus \{p,q\} \subseteq D$ in one of the following three cases:
\begin{enumerate}
\item[\textnormal{($D'_1$)}]~$K_w$ is an admissible diamond with~$p
\in D$ such that the internal angle of~$D$ at~$q \in \partial D$ is
greater than~$\pi$;
\item[\textnormal{($D''_1$)}]~$K_w$ is an admissible diamond,~$p,q \in
\partial D$, and the internal angles of~$D$ at~$p,q \in \partial D$
are greater than~$\pi$;
\item[\textnormal{($E'_1$)}]~$K_w$ is an exact kite at~$p \in D$ such
that the internal angle of~$D$ at~$q \in \partial D$ is greater
than~$\pi$.
\end{enumerate}
If the width~$w$ of~$K_w$ is sufficiently small, then
\[
\sys(M_w) = \sys(M).
\]
\end{proposition}

\begin{proof}
We will focus on the case {($E'_1$)} required for the proof of our
main theorem.  The proof in the other cases is similar.

The set~$\mathcal{C}$ of conjugacy classes in~$\pi_1(M)$ of systolic
loops of~$M=M_0$ is finite.  Denote by~$\lambda(M_w,\mathcal{C})$ the
least length of a loop of~$M_w$ from a nontrivial conjugacy class
\emph{not} in~$\mathcal{C}$.
Clearly,~$\lambda(M,\mathcal{C})>\sys(M)$.  By continuity, we still
have~$\lambda(M_w,\mathcal{C})> \sys(M) \geq \sys(M_w)$ for
sufficiently small~$w$.  Therefore, a systolic loop~\mbox{$\gamma_w
\subseteq M_w$} necessarily represents a class~$C \in \mathcal{C}$.

Consider a systolic loop~$\gamma\subseteq M$ representing the
class~$C$.  Recall that~$\gamma$ does not meet~$D$ and observe
that~$|\pi_w(\gamma)| = |\gamma|$ where~$\pi_w$ is the homotopy
equivalence~\eqref{e62b}.

Suppose~$\gamma$ does not pass through the singularity~$q\in\partial
D$.  Then the loop~$\gamma$ is disjoint from~$K_w$.  Then the
projection~$\pi_w(\gamma)$ remains a closed geodesic in~$M_w$, and
$|\pi_w(\gamma)|=|\gamma|$.  Thus the systolic loop~$\gamma_w$ and the
loop~$\pi_w(\gamma)$ are freely homotopic closed geodesics in~$M_w$.
By the flat strip theorem, we obtain~$|\gamma_w|=|\gamma|$ and
therefore~$\sys(M_w)=\sys(M)$.

Now assume~$q\in\gamma$.  Since the angle of~$D$ at~$q$ is greater
than~$\pi$, the rotation angle~$L_q$ (see Definition~\ref{r42})
of~$\gamma\subseteq M$ at $q$ from the side of~$D$ is also greater
than~$\pi$.  Note that the angle~$\measuredangle rqr'$ at~$q$ of the
excised exact kite~$K_w$ tends to zero.  Choosing
\[
\measuredangle rqr'<L_q-\pi,
\]
we ensure that the rotation angle at~$q$ is still greater than~$\pi$
for the projected loop \mbox{$\pi_w(\gamma)\subseteq M_w$}.  By the
local characterisation of geodesics, the projected loop
\mbox{$\pi_w(\gamma)\subseteq M_w$} is still a closed geodesic.  Since
$|\pi_w(\gamma)|=|\gamma|$, the loops~$\gamma_w$ and~$\pi_w(\gamma)$
are freely homotopic closed geodesics in~$M_w$.  By the flat strip
theorem, we conclude that \mbox{$|\gamma_w|=|\gamma|$} and
hence~$\sys(M_w)=\sys(M)$, as required.
\end{proof}

\subsection{Analysis of case ($E_2$)}

\begin{proposition} 
\label{prop:sys2}
Consider an admissible kite~$K_w$ with main diagonal~$[p,q]$
satisfying~\textnormal{($E_2$)}.  Namely,~$[p,q]$ is contained in the
interior of an edge of~$\partial D$ and~$K_w$ is an exact kite at~$p$
so that the part of every systolic loop passing through~$K_w$ is
parallel to~$[p,q]$.  If the width~$w$ of~$K_w$ is sufficiently small,
then
\[
\sys(M_w) = \sys(M).
\]
\end{proposition}

\begin{proof}
As in the proof of Proposition~\ref{prop:sys1b}, a systolic
loop~$\gamma_w \subseteq M_w$ necessarily represents a class~$C$ in
the set~$\mathcal{C}$ of conjugacy classes of systolic loops of~$M$.
Consider the segment~$I_w\subseteq M_w$ defined by
\[
I_w=\pi_w(K_w)=\pi_w([p,q]),
\]
where~$\pi_w\colon M\to M_w$ is the projection~\eqref{e62b}.
Let~$c_w=\gamma_w \setminus I_w$ in~$M_w$.  Consider a
systolic loop~$\gamma\subseteq M$ representing the class~$C$.

If~$\gamma$ is disjoint from the segment~$[p,q]$ (and hence from~$K_w$
if~$w$ is sufficiently small) then its projection~$\pi_w(\gamma)$
remains a closed geodesic in~$M_w$ of the same length as~$\gamma$.  By
the flat strip theorem, we conclude that~$|\gamma_w|=|\gamma|$
and~$\sys(M_w)=\sys(M)$.

Thus, we can assume that the class~$C$ contains a single (isolated)
systolic loop~$\gamma \subseteq M$, which meets the segment~$[p,q]$
and therefore must contain~$[p,q]$ by condition~\textnormal{($E_2$)}.
Now Proposition~\ref{prop:sys2} will result from the following lemma.

\begin{lemma}
Let~$\gamma_w\subseteq M_w$ be a systolic loop.
Let~$I_w=\pi_w(K_w)\subseteq M_w$ be the segment given by the image of
the kite~$K_w$.  Let~$c_w=\gamma_w \setminus I_w$.  Then
$c_w\subseteq M_w$ is a connected open segment.
\end{lemma}

\begin{proof}
Recall that the map~$\phi_w\colon M_w\to M$ of~\eqref{e65}
is~$(1+\epsilon)$-bilipschitz.  Since~$\sys(M_w) \leq \sys(M)$, the
loop~$\phi_w(\gamma_w) \subseteq M$ homotopic to~$\gamma$ is of length
at most~$(1+\epsilon)|\gamma|$.  Since the systolic loop~$\gamma
\subseteq M$ is isolated, the loop~$\phi_w(\gamma_w)$ necessarily
converges to~$\gamma$.

Since~$\gamma$ is a simple loop containing the main diagonal of~$K_w$,
the part of~$\gamma$ lying outside~$K_w\subseteq M$ necessarily
consists of a single arc for~$w$ small enough.  It follows that there
is a single subarc of~$\phi_w(\gamma_w)$ lying outside some small open
neighborhood~$U$ of~$K_w$ such that~$K_w$ is a deformation retract
of~$U$.  The image of this subarc by~$\phi_w^{-1}$ lies in an open
subarc~$\alpha_w$ of~$\gamma_w$ lying in~$M_w \setminus I_w$ with
endpoints in~$I_w$.  Since~$M$ is nonpositively curved and~$I_w$ is
convex, all the other geodesic subarcs of~$\gamma_w$ with endpoints
in~$I_w$, which lie in a small neighborhood of~$I_w$, in fact lie
in~$I_w$.  Thus,~$\alpha_w$ is the only subarc of~$\gamma_w$ lying
outside~$I_w$, that is,~$c_w = \alpha_w$.
\end{proof}

We continue with the proof of Proposition~\ref{prop:sys2}.
Let~$\sigma_w\subseteq M$ be the closure
of~$\pi_w^{-1}(c_w)$ in~$M$.  

Suppose one of the endpoints of~$\sigma_w$ is one of the main vertices
of the kite, say~$p$.  Let~$y$ be the other endpoint.  The
segment~$[p,y]\subseteq K_w$ projects to the path of~$I_w\subseteq
M_w$ connecting~$\pi_w(p)$ and~$\pi_w(y)$.  Then the
loop~$\bar\gamma_w=\sigma_w \cup [p,y] \subseteq M$ in the homotopy
class~$C$ satisfies~$|\bar\gamma_w| \leq |\gamma_w|$.  Thus,~$\sys(M)
\leq \sys(M_w)$, providing the required bound.  Therefore, we can
assume that the endpoints of~$\sigma_w$ are disjoint
from~$\{p,q\}\subseteq M$.

Suppose one of the endpoints of the path~$\sigma_w$ in~$\partial
K_w\subseteq M$ is a point other than~$r$ and~$r'$.  In such case, the
minimizing loop~$\gamma_w\subseteq M_w$ meets the interval~$I_w$
transversely at a regular (i.e., non-singular) point of~$M_w$.  It
follows that the endpoints of~$\sigma_w$ project to the same point on
the closed geodesic~$\gamma\subseteq M$.  Hence the nearest-point
projection of~$\sigma_w$ to~$\gamma$ closes up.  By the assumption of
nonpositive curvature, the projection map is distance-decreasing.
Therefore~$|\gamma_w|\geq|\sigma_w|\geq|\gamma|=\sys(M)$ in this case,
as well.

Thus we can assume that the endpoints of~$\sigma_w$ are the
points~$r,r'\in M$.  In this case also the nearest-point projection
of~$\sigma_w$ to the loop~$\gamma\subseteq M$ closes up.
Hence~$|\gamma_w|\geq|\sigma_w|\geq|\gamma|=\sys(M)$, proving the
proposition.
\end{proof}

\section{Exploiting the kite excision trick} \label{sec:exploit}

We proceed to the proof of the existence of nonpositively curved
piecewise flat locally extremal metrics on every genus~$g$ surface.  

Recall that a local infimum of the systolic area on the
space~$\mathcal{H}_g$ of nonpositively curved Riemannian metrics
(possibly with conical singularities) on a genus~$g$ surface is a real
number~$\mu > 0$ such that there exists an open
set~$\mathcal{U}\subseteq \mathcal{H}_g$ satisfying a strict
inequality
\begin{equation}
\mu = \inf_{M \in \mathcal{U}} \sigma(M) < \inf_{M \in \partial
\mathcal{U}} \sigma(M),
\end{equation}
as in Definition~$\ref{def:local}$.  Recall that~$\bar Q(g)$ is the
maximal number of systolic homotopy classes; see
Proposition~\ref{item:kissing3}.

\begin{theorem}
\label{theo:bound}
Let~$\mathcal{U}$ be an open set in the space~$\mathcal{H}_g$ of
nonpositively curved genus~$g$ surfaces (possibly with conical
singularities) defining a local infimum.  Then there exists a
nonpositively curved piecewise flat metric~$\G_0$ in~$\mathcal{U}$
with at most
\[
\mathcal{N}_0 = 20 \,\bar{Q}(g)^2 \leq 2^{25} g^4 \log^2(g)
\]
conical singularities whose systolic area is the infimum of the
systolic area of any nonpositively curved Riemannian
metric~$\G\in\mathcal{U}$ with conical singularities,
\ie,~$\sigma(\G_0) \leq \sigma(\G)$.
\end{theorem}

\begin{proof}
Let~$\mathcal{G}\in\mathcal{U}$ be a nonpositively curved Riemannian
metric with conical singularities such that
\begin{equation} \label{eq:strict}
\sigma(\G) < \inf_{\partial \mathcal{U}} \sigma.
\end{equation}
By metric approximation (see~Proposition~\ref{prop:approx}) we can
assume that~$\mathcal{G}$ is a nonpositively curved piecewise flat
metric with conical singularities.  Denote by~$N$ the number of
conical singularities of~$\G$.  By compactness
(see~Proposition~\ref{prop:compact}) and the strict
inequality~\eqref{eq:strict}, there exists a metric~$\G_1$
in~$\mathcal{U}$ with minimal systolic area among all nonpositively
curved piecewise flat metrics in~$\mathcal{U}$ with at most~$N$
conical singularities.  By Proposition~\ref{prop:nonsystolic domains}
and Lemma~\ref{lem:large}, the metric~$\G_1$ (as any nonpositively
curved piecewise flat metric on~$M$) has at most
\begin{equation}
\label{e83}
\mathcal{N} = 4 \, \bar{Q}(g)^2 \leq 2^{22} g^4 \log^2(g)
\end{equation}
special conical singularities and large
conical singularities.

\medskip

It remains to find a similar upper bound on the number of small
nonspecial conical singularities for~$\G_1$ by relying on its
local extremality among all nonpositively curved piecewise flat surfaces
of~$\mathcal{U}$ with at most~$N$ singularities.  From now on, the
surface~$M$ will be endowed with the metric~$\G_1$.  Recall that the
(small) nonspecial conical singularities lie in (the closure of) the
nonsystolic domains of~$M$.

\medskip

The next pair of lemmas exploiting the kite excision trick provide
such an upper bound.

\begin{lemma}
\label{lem:int}
Let~$M$ be a local extremum of the systolic area relative to an open
set~$\mathcal U\subseteq \mathcal{H}_g$ among all nonpositively curved
piecewise flat genus~$g$ surfaces in~$\mathcal{U}$ with at most~$N$
conical singularities as in \eqref{e83}.  Then every nonsystolic
domain~$D\subseteq M$ contains at most one small conical singularity.
\end{lemma}

\begin{proof}
We argue by contradiction.  Suppose~$p$ and~$q'$ are two conical
singularities in~$D$ with~$p$ small.  Let~$[p,q']$ be a
length-minimizing arc in the closure of~$D$ joining the two points.
We consider the following two cases.
\begin{enumerate}
\item
If~$[p,q']$ lies in the open domain~$D$, we denote by~$q$ the first
conical singularity along~$(p,q']$ from~$p$.  
\item
Otherwise, the arc~$[p,q']$ meets~$\partial D$, and the first point of
intersection of~$[p,q']$ with~$\partial D$ from~$p$ is a point,
denoted~$q$, at which~$D$ is strictly concave.
\end{enumerate}
In the second case, the angle of~$D$ at~$q$ is greater than~$\pi$,
which shows that~$q$ is a conical singularity.  

In either case, we apply the kite excision trick to~$[p,q]$ with an
exact kite~$K_w$ at~$p$, see~Definition~\ref{def:kite}, of width~$w$
small enough to satisfy~($E_1$) in the first case (when~$q$ lies
in~$D$) and ($E'_1$) in the second case (when~$q$ lies in~$\partial
D$); see~Section~\ref{sec:comparison}.  We also choose~$w$ small
enough to ensure that the resulting piecewise flat surface~$M_w$ lies
in~$\mathcal{U}$; see~Proposition~\ref{prop:local}.  By
Proposition~\ref{prop:same}, the surface~$M_w$ is nonpositively curved
and has no more conical singularities than~$M$.  By
Proposition~\ref{prop:sys1}, the systole of~$M_w$ is equal to the
systole of~$M$.  As the area of~$M_w$ is less than the area of~$M$,
this contradicts the local extremality of~$M$ among all nonpositively
curved piecewise flat genus~$g$ surfaces in~$\mathcal{U}$ with at
most~$N$ conical singularities.
\end{proof}

\begin{lemma} 
\label{lem:edge}
Let~$M$ be a local extremum of the systolic area relative to an open
set~$\mathcal U\subseteq \mathcal{H}_g$ among all nonpositively curved
piecewise flat genus~$g$ surfaces in~$\mathcal{U}$ with at most~$N$
conical singularities as in \eqref{e83}.  Then the interior of every
edge~$\mathcal{E}$ of a nonsystolic domain~$D$ of~$M$ contains at most
one small conical singularity.
\end{lemma}

\begin{proof}
We argue by contradiction.  Let~$p$ be a small conical singularity in
the interior of the edge~$\mathcal{E}$.  Let~$q$ be a conical
singularity in the interior of~$\mathcal{E}$ adjacent to~$p$.  Note
that the conical singularity~$q$ may be large.  Apply the kite
excision trick to~$[p,q]$ with an exact kite~$K_w$ at~$p$
(see~Definition~\ref{def:kite}) of width~$w$ small enough to
satisfy~($E_2$) (see~Section~\ref{sec:comparison}) and to ensure that
the resulting piecewise flat surface~$M_w$ lies in~$\mathcal{U}$;
see~Proposition~\ref{prop:local}.  We obtain a contradiction by
arguing as in the proof of Lemma~\ref{lem:int} by applying
Proposition~\ref{prop:sys2}.
\end{proof}

\begin{remark}
Technically speaking, we show stronger results in the proofs of the
two previous lemmas.  Namely, if a small conical singularity lies in a nonsystolic domain of~$M$ then this
domain contains no other conical singularity (small or large).
Similarly, if a small conical singularity lies in the interior of an
edge of~$M$ then the interior of this edge contains no other conical
singularity (small or large).
\end{remark}

We conclude the proof of Theorem~\ref{theo:bound} as follows.
Proposition~\ref{prop:nonsystolic domains} provides an upper bound on the total number of nonsystolic domains and edges (and so vertices).
Combined with Lemma~\ref{lem:int} and Lemma~\ref{lem:edge}, this shows that the surface~$M$ has at most~$3 \, \mathcal{N}$ small nonspecial conical singularities.
Along with our previous estimates on the number of special and large conical singularities, this shows that the metric~$\G_1$ has at most
\[
\mathcal{N}_0 = 5 \, \mathcal{N} = 20 \, \bar{Q}(g)^2 \leq 2^{25} g^4 \log^2(g)
\]
conical singularities.
In particular, the open set~$\mathcal{U}$ contains nonpositively curved piecewise flat metrics on~$M$ with at most~$\mathcal{N}_0$ conical singularities.

By compactness (see~Proposition~\ref{prop:compact}), and since the
systolic area of~$\G_1$ is less than~$\inf_{\partial \mathcal{U}}
\sigma$, there exists a metric~$\G_0$ in~$\mathcal{U}$ with minimal
systolic area among all nonpositively curved piecewise flat metrics
in~$\mathcal{U}$ with at most~$\mathcal{N}_0$ conical singularities.
By definition, the metric~$\G_0$ does not depend on~$\G$ (nor
on~$\G_1$) and satisfies
\[
\sigma(\G_0) \leq \sigma(\G_1) \leq \sigma(\G)
\]
for every nonpositively curved Riemannian metric~$\G$ with conical singularities in~$\mathcal{U}$.
\end{proof}

We immediately deduce the existence of locally extremal nonpositively
curved piecewise flat metrics on every genus~$g$ surface.

\begin{corollary} 
\label{coro:extremal}
Every local infimum of the systolic area on the space~$\mathcal{H}_g$
of nonpositively curved genus~$g$ surfaces (possibly with conical
singularities) is attained by a nonpositively curved piecewise flat
metric.
\end{corollary}

\section
{Shape and singularities of nonsystolic domains} \label{sec:shape}

In this section, we provide a more precise description of nonsystolic
domains of a locally extremal nonpositively curved surface~$M$ of
genus~$g$, whose existence was established in
Corollary~\ref{coro:extremal}.  We also show that the systolic part
of~$M$ is connected.

\begin{lemma} 
\label{lemma:one}
Every nonsystolic domain~$D$ of a locally extremal surface~$M$
contains at most one conical singularity.
\end{lemma}

\begin{proof}
We argue as in the proof of Lemma~\ref{lem:int}, relying now on the
local extremality of~$M$.  Assume by contradiction that there are two
conical singularities~$p$ and~$q'$ in~$D$.  If the length-minimizing
arc~$[p,q']$ joining~$p$ to~$q'$ in the closure of~$D$ lies in~$D$, we
denote by~$q$ the first conical singularity along~$(p,q']$ from~$p$.
Otherwise, the arc~$[p,q']$ meets~$\partial D$ and its first point of
intersection (from~$p$) is a conical singularity, denoted~$q$.  

In either case, take an admissible diamond~$K_w$ with diagonal~$[p,q]$
of width~$w$ small enough to satisfy~($D_1$) in the former case
and~($D'_1$) in the latter case; see~Section~\ref{sec:comparison}.
Apply the kite excision trick to~$K_w$.  The resulting piecewise flat
surface~$M_w$ may have more conical singularities than~$M$, but is
still nonpositively curved; see~Proposition~\ref{prop:same}.  Taking
the width of~$K_w$ small enough as in Proposition~\ref{prop:local}, we
can further ensure that the surface~$M_w$ lies in the open
set~$\mathcal{U}$ of~$\mathcal{H}_g$ involved in the definition of a
locally extremal nonpositively curved metric on~$M$;
see~Definition~\ref{def:local}.  By Proposition~\ref{prop:sys1}, the
systole of~$M_w$ is greater or equal to the systole of~$M$.  As the
area of~$M_w$ is less than the area of~$M$, this contradicts the local
extremality of~$M$ among all nonpositively curved piecewise flat
genus~$g$ surfaces in~$\mathcal{U}$, establishing the lemma.
\end{proof}

\begin{lemma}
\label{l92}
A nonsystolic non-simply-connected domain~$D$ of a locally extremal
surface~$M$ is necessarily convex.
\end{lemma}

\begin{proof}
Assume that~$D$ is nonconvex.  Then there is a conical
singularity~$x\in\partial D$ where the angle of~$D$ is greater
than~$\pi$.  Consider a length-minimizing noncontractible
loop~$\gamma$ based at~$x$ in the closure of~$D$, which contains
an arc~$[p,q]$ with~$p,q \in \partial D$, whose interior~$(p,q)$ lies
in~$D$.  The angles of~$D$ at the points~$p$ and~$q$ are greater
than~$\pi$, which implies that these two points are conical
singularities.  Note that the points~$p$ and~$q$ may agree.  We apply
the kite excision trick to an admissible diamond with diagonal~$[p,q]$
of width small enough to satisfy~($D'_1$), or~($D''_1$) if~$p=q$, and
derive a contradiction as in the proof of Lemma~\ref{lemma:one}.  This
shows that such a domain~$D$ must be convex.
\end{proof}

\begin{proposition} 
\label{prop:one}
Every nonsystolic domain~$D$ of a locally extremal surface~$M$ is
homeomorphic to a disk.
\end{proposition}

\begin{proof}
Arguing by contradiction, we suppose that~$D\subseteq M$ is nonsimply
connected.  By Lemma~\ref{l92},~$D$ must be convex.  Assume that~$D$
contains a conical singularity~$p$.  By Lemma~\ref{lemma:one}, this is
the only conical singularity in~$D$.  Since~$D$ is convex (and
nonsimply connected), there is a length-minimizing noncontractible
loop~$\gamma$ based at~$p$ lying in~$D$.  We apply the kite excision
trick to an admissible diamond with diagonal the geodesic
arc~$\gamma$, starting and ending at~$p$, of width small enough to
satisfy~($D_1$), and derive a contradiction as in the proof of
Lemma~\ref{lemma:one}.  This shows that~$D$ has no conical
singularity.

By the Gauss--Bonnet formula for surfaces with boundary, the Euler
characteristic of the (orientable) flat surface~$D\subseteq M$ with
convex boundary is nonnegative.  This implies that~$D$ is a flat
cylinder (by assumption, it is not a disk).  In this case, it also
follows from the Gauss--Bonnet formula that the cylinder~$D=S^1\times
I$ has geodesic boundary components.  Since the cylinder is
nonsystolic, its boundary loops are systolic geodesics, and any closed
geodesic not parallel to the boundary must have length strictly
greater than~$\sys(M)$.  Therefore we can slightly shrink the
height~$I$ of the cylinder without affecting the systole of~$M$, and
respecting the condition of nonpositive curvature on~$M$.  This
contradicts the local extremality of the surface.  Hence, the
domain~$D$ is simply connected and so is a disk.
\end{proof}

As a consequence of Proposition~\ref{prop:one}, the systolic part
of~$M$, defined as the union of its systolic loops, is obtained by
removing finitely many open disks from the surface.  In particular, we
obtain the following corollary.

\begin{corollary}
The systolic part of a locally extremal nonpositively curved surface
is path-connected.
\end{corollary}

\end{document}